\documentclass[%draft,
reqno]{amsart} 
\usepackage[T1]{fontenc}
\usepackage[english]{babel}
\usepackage{amssymb,bbm,enumerate,amsmath,mathtools}
\usepackage{color}
\usepackage[linktocpage=true,colorlinks=true, linkcolor=blue, citecolor=red, urlcolor=green]{hyperref}

\usepackage{xypic}

\usepackage{tikz}

\renewcommand{\phi}{\varphi}
\renewcommand{\ker}{\Ker}

\newcommand{\mc}[1]{\mathcal{#1}}
\newcommand{\mf}[1]{\mathfrak{#1}}
\newcommand{\mb}[1]{\mathbb{#1}}

\DeclareMathOperator{\End}{End}

\DeclareMathOperator{\res}{Res}

\DeclareMathOperator{\Ker}{Ker}

\DeclareMathOperator{\Res}{Res}

\newcommand{\vac}{|0\rangle}

\newcommand{\Kh}{\mathbb{K}[[h]]}

\newcommand{\spd}[1]{\partial_{#1}}
\newcommand{\gse}[2]{\iota_{#1, #2}}

\newcommand{\Sp}[2]{\mc{S}^{#1 #2}}

\theoremstyle{plain}
\newtheorem{theorem}{Theorem}[section]
\newtheorem{lemma}[theorem]{Lemma}
\newtheorem{proposition}[theorem]{Proposition}
\newtheorem{corollary}[theorem]{Corollary}

\theoremstyle{definition}
\newtheorem{definition}[theorem]{Definition}
\newtheorem{example}[theorem]{Example}

\theoremstyle{remark}
\newtheorem{remark}[theorem]{Remark}

\numberwithin{equation}{section}

\definecolor{light}{gray}{.9}

\begin{document}

\title{On the structure of quantum vertex algebras}

\author{Alberto De Sole}
\address{Dipartimento di Matematica, Sapienza Universit\`a di Roma,
P.le Aldo Moro 2, 00185 Rome, Italy}
\email{desole@mat.uniroma1.it}
\urladdr{http://www1.mat.uniroma1.it/~desole}
%\thanks{The first author is supported in part by a Simons Foundation grant 584741}
%
\author{Matteo Gardini}
\address{Dipartimento di Matematica, Sapienza Universit\`a di Roma,
P.le Aldo Moro 2, 00185 Rome, Italy}
\email{gardini@mat.uniroma1.it}
\urladdr{http://www1.mat.uniroma1.it/~gardini/}
\author{Victor G. Kac}
\address{Department of Mathematics, MIT,
77 Massachusetts Ave., Cambridge, MA 02139, USA}
\email{kac@math.mit.edu}
\urladdr{http://www-math.mit.edu/~kac/}

%\subjclass{
%Primary 17B63; 
%Secondary 17B69, 17B80, 37K30, 17B08
%}

%%%%%%%%

\begin{abstract}
A definition of a quantum vertex algebra, which is a deformation of a vertex algebra, was proposed by Etingof and Kazhdan in 1998. In a nutshell, a quantum vertex algebra is a braided state-field correspondence which satisfies associativity and braided locality axioms. We develop a structure theory of quantum vertex algebras, parallel to that of vertex algebras. 
In particular, we introduce braided n-products for a braided state-field correspondence and prove for quantum vertex algebras a version of the Borcherds identity.
\end{abstract}


\maketitle

%\tableofcontents

\pagestyle{plain}

%%%%%%%%%%%%%%%%%%%%%%%%%%%%%%%%%%%%%%%
\section{Introduction}\label{sec:1}

Let $V$ be a module over a unital commutative associative ring $\mb K$. The basic objects of study in the theory of vertex algebras are products on $V$ with values in Laurent series over $V$, i.e. $\mb K$-linear maps
\begin{equation}
\label{VK1.1}
Y: V \otimes V \to V((z))\,\,, \,\,\,\, a\otimes b \mapsto Y(z)(a \otimes b).
\end{equation}
A product \eqref{VK1.1} is called \emph{unital} if there exists a non-zero element 
$\vac \in V$, called the \emph{vacuum vector}, such that, for any $a \in V$, one has
\begin{equation}
\label{VK1.2}
Y(z)(\vac \otimes a) = a,\, Y(z)(a \otimes \vac) \equiv a\, mod\, zV[[z]].
\end{equation}
Condition \eqref{VK1.2} is called the \emph{vacuum axiom}. 
Given a unital product \eqref{VK1.1} on $V$, one defines 
the \emph{translation operator} $T$ on $V$ by
\begin{equation}
\label{VK1.3}
Ta = \spd{z} Y(z)(a \otimes \vac)_{|z=0}\,\,,\,\,\,\, a \in V.
\end{equation}
The product \eqref{VK1.1} is called \emph{translation covariant} if the following two conditions hold:
\begin{align}
\label{VK1.4} TY(z)(a \otimes b) - Y(z)(a \otimes Tb) = \spd{z} Y(z)(a \otimes b),\\
\label{VK1.5} Y(z)(Ta \otimes b) = \spd{z} Y(z)(a \otimes b).
\end{align}
Note that these two conditions imply that $T$ is a derivation of the product \eqref{VK1.1} and, if this holds, conditions \eqref{VK1.4} and \eqref{VK1.5} are equivalent.

A unital translation covariant product \eqref{VK1.1} is called a \emph{state-field correspondence}.\\
A \emph{vertex algebra} is a state-field correspondence on a vector space $V$ over a field $\mb K$ of characteristic $0$ which satisfies the \emph{locality axiom}
\begin{equation}
\label{VK1.6}
\begin{split}
(z-w)^N Y(z) (1 \otimes Y(w))(a \otimes b \otimes c) = (z-w)^N Y(w) (1 \otimes Y(z))(b \otimes a \otimes c)
\end{split}
\end{equation}
for some non-negative integer $N$ depending on elements $a, b \in V$.

Introducing the quantum fields $Y(a,z)$, $a \in V$, by $Y(a,z)b = Y(z)(a \otimes b)$, $b \in V$, we obtain a more familiar form of equations \eqref{VK1.1}-\eqref{VK1.6} respectively 
(cf. \cite{K}, \cite{K15}):
\begin{enumerate}
\item[\eqref{VK1.1}'] $Y(a,z) = \sum_{n \in \mb Z} a_{(n)} z^{-n-1}$, where $a_{(n)} \in \End V$,
\item[\eqref{VK1.2}'] $Y(\vac, z) = I_V$, $Y(a,z)\vac \in a + z V[[z]]$,
\item[\eqref{VK1.3}'] $Ta = \spd{z} Y(a,z)\vac\,\big|_{z=0}$,
\item[\eqref{VK1.4}'] $[T,Y(a,z)] = \spd{z}Y(a,z)$,
\item[\eqref{VK1.5}'] $Y(Ta,z) = \spd{z} Y(a,z)$,
\item[\eqref{VK1.6}'] $(z-w)^N [Y(a,z), Y(b,w)] = 0$.
\end{enumerate}
Then the map $a \mapsto Y(a,z)$ is indeed a correspondence between states $a \in V$ and quantum fields $Y(a,z)$. This definition of vertex algebra was given in \cite{K} (see also \cite{K15}), where its equivalence to the original definition, given by Borcherds \cite{B} in 1986 (via Borcherds identity), was established.

A quantum deformation of the notion of a vertex algebra is obtained by taking for the space of states $V$ a topologically free module over $\mb K[[h]]$ and considering a ``topological'' state-field correspondence satisfying a ``deformed'' axiom of locality which induces locality \eqref{VK1.6} on $V/hV$ (hence induces on V/hV a structure of a vertex algebra). Following ideas of \cite{FR}, Etingof and Kazhdan in \cite{EK5} defined such a ``deformed'' locality by introducing a $\mb K[[h]]$-linear map, called \emph{braiding},
\begin{align*}
\mc S: V \widehat{\otimes} V \to V \widehat{\otimes} V \widehat{\otimes} (\mb K((z))[[h]])
\end{align*}
(where $\widehat{\otimes}$ stands for the $\mb K[[h]]$-completed tensor product), which is $\equiv 1$ mod $h$ and is a shift-invariant unitary solution of the quantum Yang-Baxter equation (see Definition \ref{BVADef}). To get the $\mc S$-locality they inserted $\mc S$ in the LHS of \eqref{VK1.6} as follows (cf. \cite{FR})
\begin{equation}
\label{VK1.7}
\begin{split}
(z-w)^N Y(z) (1 \otimes Y(w))\mc S^{1 2}(z-w) (a \otimes b \otimes c)\\
= (z-w)^N Y(w) (1 \otimes Y(z))(b \otimes a \otimes c)\ mod\ h^M
\end{split}
\end{equation}
for all $a, b, c \in V$, $M \in \mb Z_{\geq 0}$, and $N \in \mb Z_{\geq 0}$ depending on $a,b$ and $M$.

The resulting notion is called a \emph{braided vertex algebra}. Note that if $V$ is a braided vertex algebra, obviously $V/hV$ is a vertex algebra over $\mb K$. Hence such $V$ is called a \emph{quantization} of the vertex algebras $V/hV$.

It is well known that vertex algebras satisfy the associativity relation
\begin{equation}
\label{VK1.8}
(z+w)^N Y(w)(Y(z) \otimes 1) (a \otimes b \otimes c) = (z+w)^N \gse{z}{w} Y(z+w)(1 \otimes Y(w))(a \otimes b \otimes c)
\end{equation}
for some non-negative integer $N$ (see e.g. {\cite[Prop.4.1(b)]{BK}}), which implies that, in case $T=0$, vertex algebras are associative algebras. However, this is not the case for braided vertex algebras.

Etingof and Kazhdan in \cite{EK5} rectified the definition of a braided vertex algebra by imposing on the braiding $\mc S$ the hexagon relation (see equation \eqref{HexagonRelation}), and proved that a braided vertex algebra with braiding $\mc S$ satisfying the hexagon relation, satisfies the associativity relation \eqref{VK1.8}. They called such a braided vertex algebra a \emph{quantum vertex algebra}.

The main result of \cite{EK5} is a construction of a quantum vertex algebra which is a non-trivial quantization of the universal affine vertex algebra $V^k(\mathfrak{sl}_N)$. They also discuss quantum vertex algebras with $T=0$ (which are associative algebras with braiding) and the quasiclassical limit of a quantum vertex algebra (which is a vertex algebra with some additional structure).

In the present paper we develop a structure theory of braided and quantum vertex algebras, keeping in mind the structure theory of vertex algebras. For the former we first recall some of the results of the latter.

The following theorem, 
collecting results from \cite{K}, \cite{LL}, \cite{BK} and \cite{L03},
gives several equivalent characterizations of a vertex algebra.
\begin{theorem}
\label{CharacterizationVAsIntro}
Let $V$ be a vector space over a field $\mb K$ of characteristic $0$ and let $Y:\, V\otimes V\to V((z))$ be a state-field correspondence.
Then the following conditions are equivalent ($a,b,c \in V$):
\begin{enumerate}[(i)]
\item $V$ is a vertex algebra.
\item The following Jacobi identity holds:
\begin{equation}
\label{JacobiIntro}
\begin{split}
& \gse{z}{w} \delta(x, z-w) Y(z) \big( 1 \otimes Y(w) \big)(a \otimes b \otimes c) \\
& - \gse{w}{z} \delta(x, z-w) Y(w) \big( 1 \otimes Y(z) \big) (b \otimes a \otimes c) \\
&= \gse{z}{x} \delta(w, z-x) Y(w)\big( Y(x) \otimes 1 \big) (a \otimes b \otimes c)\,.
\end{split}
\end{equation}
\item The following two equations hold: 
the associativity relation ($N\gg0$)
\begin{equation}
\begin{split}
& (z+w)^N Y(w) \big(Y(z) \!\otimes\! 1 \big) (a \!\otimes\! b \!\otimes\! c) \\
& = (z+w)^N \gse{z}{w} Y(z+w) \big( 1 \!\otimes\! Y(w) \big) (a \!\otimes\! b \!\otimes\! c)
\,,
\end{split}
\end{equation}
and the skewsymmetry relation
\begin{equation}
\label{YYopIntro}
Y(z)(a \otimes b) = e^{zT} Y(-z)(b \otimes a) =: Y^{op}(z)(a \otimes b)
\,.
\end{equation}
\item The Borcherds identity holds for any $n \in \mb Z$:
\begin{equation}
\label{BorcherdsIdentitiesIntro}
\begin{split}
&\gse{z}{w}(z-w)^n
Y(z) \big(1\otimes Y(w) \big)(a \otimes b \otimes c ) \\
& -\gse{w}{z}(z-w)^n
Y(w) \big( 1\otimes Y(z) \big)(b \otimes a \otimes c )\\
&= \sum_{j \geq 0} Y(w)( a_{(n+j)} b \otimes c ) \frac{\spd{w}^j \delta(z,w)}{j!}.
\end{split}
\end{equation}
\item 
The skewsymmetry $Y = Y^{op}$,
and the following $n$-product identities holds:
\begin{equation}
\label{nproductsIdentitiesIntro}
Y(a,z)_{(n)} Y(b,z) = Y(a_{(n)} b, z)\,\,,\,\,\,\, n \in \mb Z
\,.
\end{equation}
\end{enumerate}
\end{theorem}
Here and below $\delta(z,w)$ denotes the formal $\delta$-function
$\delta(z,w) = \sum_{k \in \mathbb{Z}} z^{-k-1} w^k$,
and $\gse{z}{w}$ denotes the geometric series expansion for $|z| > |w|$.
In condition (v), 
the \emph{$n$-product} of quantum fields $a(z), b(z): V \rightarrow V((z))$ is defined by
\begin{equation}
\label{nproductsIntro}
a(z)_{(n)} b(z) = Res_x \big( a(x) b(z) \gse{x}{z}(x-z)^n - b(z) a(x) \gse{z}{x}(x-z)^n \big).
\end{equation}

A vertex algebra is called \emph{commutative} if the locality axiom \eqref{VK1.6} holds for $N=0$:
\begin{equation}
\label{CommutativityIntro}
Y(z) \big( 1 \otimes Y(w) \big)(a \otimes b \otimes c) = Y(w) \big( 1 \otimes Y(z) \big)(b \otimes a \otimes c)
\,,\,\,\,\, a,b,c\in V\,.
\end{equation}
Equivalently, this means that $[Y(a,z), Y(b,w)] = 0$ for any $a,b \in V$.

The following simple result can be found in \cite{B} and \cite{K}.
\begin{theorem}
\label{CommutativityThmIntro}
A vertex algebra is commutative if and only if $Y(z)(a \otimes b) \in V[[z]]$ for any $a, b \in V$.
Moreover, in this case,
the $(-1)$-product $a_{(-1)}b$ is commutative, associative, 
unital (with unity $\vac$), differential (with derivative $T$),
and the state-field correspondence $Y$ is given by:
$Y(z) (a \otimes b) = (e^{zT}a)_{(-1)} b$.
\end{theorem}

The main goal of the present paper is to describe the quantum analogue
of Theorem \ref{CharacterizationVAsIntro}.
We also describe a quantum analogue of Theorem \ref{CommutativityThmIntro}.

First, the quantum analogue of the Jacobi identity \eqref{JacobiIntro} is
the following \emph{$\mc{S}$-Jacobi identity} (cf. \cite{L10}):
\begin{equation}
\label{SJacobiIntro}
\begin{split}
&\gse{z}{w} \delta(x, z-w) Y(z) \big( 1\otimes Y(w) \big)(a \otimes b \otimes c)\\
&- \gse{w}{z} \delta(x, z-w) Y(w) \big( 1\otimes Y(z) \big)\big(\mc{S}(w-z)(b \otimes a) \otimes c \big)\\
&= \gse{z}{x} \delta(w, z -x) Y(w) \big( Y(x) \otimes 1 \big) (a \otimes b \otimes c).
\end{split}
\end{equation}
The associativity relation \eqref{VK1.8} remains the same
for quantum vertex algebras, but it has to be understood modulo $h^M$
(with the exponent $N$ depending on $M$).
The skewsymmetry relation \eqref{YYopIntro} becomes the quantum skewsymmetry 
\begin{equation}\label{eq:ginocchio1}
Y\mc{S} = Y^{op}\,,
\end{equation} 
where $Y^{op}$ is defined by \eqref{YYopIntro}
(see {\cite[Lemma 1.2]{EK5}}).
The quantum analogue of the Borcherds identity \eqref{BorcherdsIdentitiesIntro}
is the \emph{quantum Borcherds identity} ($n \in \mb Z$)
\begin{equation}
\label{qBorcherdsIdentitiesIntro}
\begin{split}
& Y(z) \big( 1\otimes Y(w) \big) \big( \gse{z}{w} \mc{S}(z-w)(a \otimes b) \otimes c \big) \gse{z}{w}(z-w)^n\\
& -Y(w) \big( 1\otimes Y(z) \big) (b \otimes a \otimes c)\gse{w}{z}(z-w)^n\\
&= \sum_{j \in \mb Z_{\geq0}} Y(w)\big(a_{(n+j)}^{\mc{S}}b \otimes c \big)\ \frac{\spd{w}^j \delta(z,w)}{j!}
\,,
\end{split}
\end{equation}
where the products $a^{\mc S}_{(n)}b$ are defined as the Fourier coefficients 
of $Y(z)\mc S(z)(a\otimes b)$, and the quantum analogue of the $n$-products \eqref{nproductsIntro}
is as follows ($n \in \mb Z$)
\begin{equation}
\label{qnproductsDefIntro}
\begin{split}
(Y(z)_{(n)}^{\mc{S}} Y(z))(a\!\otimes\! b\!\otimes\!c)
& = Res_x \Big(
\gse{x}{z}(x-z)^n
Y(x) \big( 1\otimes Y(z) \big) \mc{S}^{12}(x-z)(a\!\otimes\! b\!\otimes\!c) \\
& - \gse{z}{x}(x-z)^n 
Y(z) \big( 1\otimes Y(x) \big)(b \otimes a \otimes c)
\Big)
\,.
\end{split}
\end{equation}
Then, the \emph{quantum} $n$-\emph{product identity} is ($n \in \mb Z$)
\begin{equation}
\label{qnproductIdentitiesIntro}
\Big( Y(z)_{(n)}^{\mc{S}} Y(z)\Big)(a\otimes b\otimes c) 
= Y(z)\big(a_{(n)}^{\mc{S}}b \otimes c \big).
\end{equation}
We can now formulate the quantum analogue of Theorem \ref{CharacterizationVAsIntro}
(cf. Theorem \ref{thm:final}):
\begin{theorem}\label{thm:final-intro}
Let $V$ be a topologically free $\mb K[[h]]$-module with a topological state-field correspondece $Y: V \widehat{\otimes} V \to V((z))$ and braiding $\mc S$. 
Then the following conditions are equivalent:
\begin{enumerate}[(i)]
\item $V$ is a braided vertex algebra 
and the associativity relation \eqref{VK1.8} holds modulo $h^M$;
\item the $\mc{S}$-Jacobi identity \eqref{SJacobiIntro} holds;
\item the associativity relation \eqref{VK1.8} and the quantum skewsymmetry \eqref{eq:ginocchio1} hold; 
\item the quantum Borcherds identity \eqref{qBorcherdsIdentitiesIntro} holds;
\item the quantum $n$-product identities \eqref{qnproductIdentitiesIntro} 
and the $\mc{S}$-locality \eqref{VK1.7} hold.
\end{enumerate}
\end{theorem}
A topological state-field correspondence $Y$ with braiding $\mc S$ is a \emph{quantum vertex algebra}
if one of the equivalent conditions of Theorem \ref{thm:final-intro} holds.

Note that in the present paper we give a definition of a braided and a quantum vertex algebra $V$,
not requiring that the braiding $\mc S$ satisfies the shift condition,
unitarity and the quantum Yang-Baxter equation,
and replacing the hexagon relation by associativity of $V$.
In our structure theory of quantum vertex algebras
these conditions on $\mc S$ are never used.
We show that all these conditions hold automatically modulo the kernel of $Y$
(in \cite[Prop.1.11]{EK5} they are proved under a ``non-degeneracy'' assumption).
%It follows from \cite[Prop.1.11]{EK5} and Proposition \ref{prop:non-deg}
%that our definition coincides with theirs
%when $V/hV$ is a ``non-degenerate'' vertex algebra.

We also find a quantum analogue of Theorem \ref{CommutativityThmIntro}.
In order to state it, we define the following $\mc{S}$-\emph{commutativity}
(cf. \eqref{CommutativityIntro}):
\begin{equation}
\label{S-commutativityIntro}
Y(z) \big( 1\otimes Y(w) \big)\big(\gse{z}{w}\mc{S}(z-w)(a \otimes b) \otimes c \big) = Y(w) \big( 1\otimes Y(z) \big)(b \otimes a \otimes c).
\end{equation}
We have, as immediate consequence 
of Theorems \ref{ScommutativeBVAHolomBVA}
and \ref{thm:scomm-bva},
\begin{theorem}
\label{ScommutativeBVAHolomBVAIntro}
Let $V$ be a quantum vertex algebra with a fixed braiding $S(z)$. 
The $\mc{S}$-commutativity
\eqref{S-commutativityIntro} holds if and only if 
$Y(z)(a \otimes b) \in V[[z]]$ for every $a, b \in V$.
Moreover, in this case the $(-1)$-product 
$(-_{(-1)}-):\,V\otimes V\to V$, $a\otimes b\mapsto a_{(-1)}b$,
is a unital (with unity $\vac$), associative, differential (with derivation $T$) product,
satisfying the following ``quantum commutativity'' relation
(cf. equation \eqref{ScommutativeBVA+AssociativityCommRelations})
$$
b_{(-1)}a
=
(-_{(-1)}-)
Res_z z^{-1} 
(e^{zT} \otimes 1) \mc{S}(z) (a \otimes b)
\,.
$$
Furthermore, the state-field correspondence $Y(z)$ is given by
$$
Y(z)(a\otimes b)=(e^{zT}a)_{(-1)}b
\,.
$$
\end{theorem}

A powerful tool for construction of vertex algebras is the existence or the extension theorem 
(see {\cite[Thm.4.5]{K}} or {\cite[Thm.1.5]{DSK06}}, \cite{K15}). 
Unfortunately we were unable to find its quantum analogue. 
Thus, construction of quantum vertex algebras remains a difficult problem.

Interesting problems are to construct non-trivial quantizations of commutative vertex algebras
and of affine vertex algebras associated to any simple Lie algebra.
Beyond \cite{EK5}, some examples were constructed in \cite{JKMY} and in \cite{BJK} respectively.

\medskip

The paper is organized as follows.
Section \ref{FieldAlgebrasVAsSection} is devoted to vertex algebras.
We start by reviewing in Section \ref{CalculusOfFormalDistributionsSection} 
the basic calculus of formal distributions,
in Section \ref{FieldAlgebrasSection} the properties of field algebras,
and in Section \ref{sec:2.3} the definition of vertex algebras.
The remaining subsections contain
several characterizations of vertex algebras given by Theorem \ref{CharacterizationVAs},
and the description of commutative vertex algebras.
Next, we switch in Section \ref{sec:3} to quantum vertex algebras.
After introducing the necessary $h$-adic topology,
we give the definitions of braided vertex algebras (Definition \ref{BVADef})
and quantum vertex algebras (Definition \ref{QVADef}).
We show in this Section that, modulo the kernel of $Y$,
all the conditions imposed in \cite{EK5} on $\mc S$
do hold (Proposition \ref{prop:non-deg}),
and the hexagon relation follows from associativity (Proposition \ref{July2017ThmRem1}).
We also prove, within this section, some preliminary results on braided and
quantum vertex algebras, and the quantum analogue 
of Goddard's uniqueness Theorem 
(cf. Proposition \ref{BraidedGoddard}).
In Section \ref{S-commutativeBVAs} we consider the special case
of $\mc S$-commutative braided and quantum vertex algebras,
see Theorems \ref{ScommutativeBVAHolomBVA} and \ref{thm:scomm-bva}.
Finally, in Section \ref{sec:5}
we introduce quantum $n$-products and
prove the quantum Borcherds identity,
see Theorem \ref{associativeBVABorcherdsIdentities}.
We conclude by proving the main characterizations of quantum vertex algebras,
see Theorem \ref{thm:final}.
This paper is based on the Ph.D thesis of the second author \cite{Gar}.

Throughout the paper all vector spaces, tensor products, hom's, etc.
are over a field $\mb F$ of characteristic zero, unless otherwise specified.

\medskip

%%%
\emph{Acknowledgments}
We are deeply grateful to Pavel Etingof for discussions and for explaining us 
the whole theory of quantum vertex algebras.
The research was partially conducted during the authors visits 
to both MIT and Sapienza University of Rome;
we are grateful to both these institutions for their kind hospitality.
The first author was partially supported 
by the national PRIN fund n.\ 2015ZWST2C$\_$001
and the University funds n. RM116154CB35DFD3 and RM11715C7FB74D63.
The second author was partially supported by the grant UMI-MIT
and an INdAM GNSAGA grant.
The third author was partially supported by the Bert and Ann Kostant fund.

%%%%%%%%%%%%%%%%%%%%%%%%%%%%%%%%%%%%%%%
\section{Field algebras and vertex algebras}
\label{FieldAlgebrasVAsSection}

In this section we review the definitions of a field algebra
and of a vertex algebra, following \cite{BK,K} (see also \cite{K15}
for a more recent exposition).

%%%
\subsection{Calculus of formal distributions}
\label{CalculusOfFormalDistributionsSection}

Given a vector space $V$,
we let $V[[z, z^{-1}]]$ be the space of bilateral formal power series with coefficients in $V$; they are called formal distributions.
A \emph{quantum field} over $V$ is a formal distribution $a(z)\in (\End V)[[z,z^{-1}]]$ with coefficients in $\End V$, such that $a(z)v\in V((z))$ for every $v\in V$. Hereafter $V((z)) = V[[z]][z^{-1}]$ stands for the space of Laurent series with coefficients in $V$.

Recall that the \emph{formal delta distribution} $\delta(z,w)$ is a formal distribution in $z$ and $w$ with coefficients in $\mb K$ defined as follows:
\begin{equation}\label{eq:delta}
\delta(z,w) 
= 
\sum_{m \in \mathbb{Z}} z^{-m-1}w^m
\,.
\end{equation}
It can be obtained as
\begin{equation}\label{eq:iota}
\delta(z,w) = \gse{z}{w} \frac{1}{z-w}-\gse{w}{z}\frac{1}{z-w}
\,,
\end{equation}
where $\gse zw$ (resp. $\gse wz$) denotes the geometric series expansion
in the domain $|z|>|w|$ (resp. $|w|>|z|$).
Recall also that,
for an arbitrary formal distribution $a(z)$, we have
\begin{equation}\label{eq:res}
Rez_z \big( a(z) \delta(z,w) \big) = a(w)
\,,
\end{equation}
where $\Res_z$ denotes the coefficient of $z^{-1}$.

\begin{definition}\label{def:local}
A formal distribution (in two variables with coefficients in $V$) $a(z,w) \in V[[z^{\pm 1}, w^{\pm 1}]]$ is called \emph{local} 
if there exists $N \in \mb Z_{\geq 0}$ such that $(z-w)^N a(z,w) = 0$.
A pair $a(z), b(z) \in (\End V)[[z, z^{-1}]]$ of formal distributions is called \emph{local} if there exists $N \in\mb Z_{\geq 0}$ such that
$$
(z-w)^N [a(z),b(w)] = 0
\,,
$$
while they are \emph{local on} $v \in V$ if there exists $N \in\mb Z_{\geq 0}$ such that
$$
(z-w)^N [a(z),b(w)]v = 0
\,.
$$
\end{definition}
\begin{theorem}[Decomposition Theorem {\cite[Cor.2.2]{K}}]
\label{KDecThm}
Any local formal distribution $a(z,w) \in V[[z^{\pm 1}, w^{\pm 1}]]$ can be uniquely decomposed as
\begin{equation}
a(z,w) = \sum_{j = 0}^N c^j(w) \frac{\partial^{j}_w \delta(z,w)}{j!},
\end{equation}
where $c^j(w) \in V[[w^{\pm 1}]]$, and one has
\begin{displaymath}
c^j(w) = Res_{z} \big( (z-w)^j a(z,w) \big).
\end{displaymath}
\end{theorem}

\begin{lemma}[{\cite[Lem.2.1]{L03}}]
\label{LiLem.2.1}
Let $a(z,w) \in V((z))((w))$, $b(z,w) \in V((w))((z))$, $c(z,w) \in V((w))((z))$.
Then
\begin{equation}
\label{LiJacobi}
\begin{split}
&\gse{z}{w}\delta(x,z-w)\ a(z,w) - \gse{w}{z} \delta(x,z-w)b(z,w) = \gse{z}{x} \delta(w,z-x)\ c(x,w)
\end{split}
\end{equation}
if and only if 
\begin{equation}\label{maratona}
(z-w)^N a(z,w) = (z-w)^N b(z,w)
\,\text{ and }\,
(x+w)^N \gse{x}{w} a(x+w,w) = (x+w)^N c(x,w)
\end{equation}
for some $N\gg0$.
\end{lemma}

In Section \ref{sec:5} we will need and we will give a proof of an analogue of Lemma \ref{LiLem.2.1}
in the context of a topologically free $\mb K[[h]]$-module,
see Lemma \ref{LiLem.2.1b}.

%%%%%%%%%%%%%%%%%%%%%%%%%%%%%%%%%
\subsection{Field algebras}
\label{FieldAlgebrasSection}
In this subsection we recall the definition of a field algebra and its properties following \cite{BK}.
\begin{definition}[\cite{BK}]
\label{pointedVectorSpaceWithStateFieldCorrespondenceDef}
A \emph{state-field correspondence} on a pointed vector space $(V, \vac)$ 
is a linear map $Y: V \otimes V \rightarrow V ((z))$, 
$a \otimes b \mapsto Y(z)(a \otimes b)$,
satisfying
\begin{enumerate}[(i)]
\item (vacuum axioms) $Y(z)(\vac \otimes a) = a, Y(z)(a \otimes \vac) \in a + V[[z]]z$;
\item (translation covariance 1) 
$TY(z)(a \otimes b) - Y(z) (a \otimes Tb) 
= \spd{z}Y(z)(a \otimes b)$, where $T(a):=\spd{z}Y(z)(a\otimes\vac)|_{z=0}$
(the translation operator);
\item (translation covariance 2) 
$Y(z)(Ta \otimes b) = \spd{z}Y(z)(a \otimes b)$.
\end{enumerate}
\end{definition}
\noindent
With an abuse of notation, we shall at times denote by $Y$ also the map 
$Y:\,V\to\End[[z,z^{-1}]]$, $a\mapsto Y(a,z)=\sum_{k \in \mathbb{Z}} a_{(k)} z^{-k-1}$, 
such that $Y(a,z)b=Y(z)(a\otimes b)$.
Note that $Y(a,z)$ is a \emph{quantum field}, i.e. $Y(a,z)b \in V((z))$ for any $b \in V$.
\begin{proposition}[see e.g. {\cite[Prop.2.7]{BK}}]
\label{BKProp.2.7}
If $Y: V \otimes V \rightarrow V ((z))$ satisfies conditions (i) and (ii) 
of Definition \ref{pointedVectorSpaceWithStateFieldCorrespondenceDef},
then
\begin{enumerate}[(a)]
\item \label{BKProp.2.7p1} $Y(z)(a \otimes \vac) = e^{zT}a$;
\item \label{BKProp.2.7p2} $e^{wT}Y(z) \left( 1 \otimes e^{-wT} \right) = \gse{z}{w} Y(z+w)$.
\end{enumerate}
If, moreover, $Y$ is a state-field correspondence, then
\begin{enumerate}[(a)]
\setcounter{enumi}{2}
\item \label{BKProp.2.7p3} $Y(z) \big(e^{wT} \otimes 1\big) = \gse{z}{w} Y(z+w)$
\end{enumerate}
\end{proposition}

\begin{proposition}[{\cite[Prop.2.8]{BK}}]
\label{BKProp.2.8}
Given a state-field correspondence $Y$, define
\begin{equation}\label{eq:Yop}
Y^{op}(z)(a \otimes b) = e^{zT}Y(-z)(b \otimes a).
\end{equation}
Then $Y^{op}$ is also a state-field correspondence.
\end{proposition}

A key ingredient for the proof of Theorem \ref{BKThm.7.3} below
is the following result, which we shall use in Section \ref{sec:5}, 
for the proof of Theorem \ref{July2017Thm} in Section \ref{sec:5}
(the ``quantum analogue'' of Theorem \ref{BKThm.7.3}).
\begin{lemma}[{\cite[Lem.3.8]{BK}}]
\label{BKLem.3.8}
Let $X$ and $Y$ be two state-field correspondences, and let $a,b,c \in V$ be such that there exists $N \geq 0$ such that
\begin{align*}
(z-w)^N Y(z) \big( 1 \otimes X(w) \big)(a \otimes c \otimes b) = (z-w)^N X(w) \big( 1 \otimes Y(z)\big)(c \otimes a \otimes b).
\end{align*}
Then
\begin{align*}
(z-w)^N Y(z) \big( 1 \otimes X(w) \big)(a \otimes c \otimes Tb) = (z-w)^N X(w) \big( 1 \otimes Y(z)\big)(c \otimes a \otimes Tb).
\end{align*}
\end{lemma}

\begin{definition}
Let $(V,\vac)$ be a pointed vector space and let $Y$ be a state-field correspondence. One says that $Y$ satisfies the \emph{associativity relation} if, for any $a, b, c \in V$, there exists $N \in\mb Z_{\geq 0}$ such that the following equation holds:
\begin{equation}
\label{fAssociativity}
(z+w)^N Y(w)\big( Y(z) \otimes 1 \big)(a \otimes b \otimes c) = (z+w)^N \gse{z}{w} Y(z+w) \big( 1 \otimes Y(w) \big)(a \otimes b\otimes c).
\end{equation}
\end{definition}

\begin{proposition}[{\cite[Prop.4.1]{BK}}]
\label{BKProp.4.1}
Let $Y$ be a state-field correspondence on a pointed vector space $(V,\vac)$. Then $Y$ satisfies the associativity relation \eqref{fAssociativity} if and only if all pairs $\big( Y(a,z), Y^{op}(b,w) \big)$ are local on each $c \in V$
(cf. Definition \ref{def:local}).
\end{proposition}

\begin{definition}
\label{fieldAlgebraDef}
A pointed vector space $(V,\vac)$ with a state-field correspondence $Y$ is called a \emph{field algebra} if the map $Y$ satisfies the associativity relation \eqref{fAssociativity}.
\end{definition}

The following proposition follows from \cite[Prop.4.7a]{BK}, see \cite[Rem.4.8]{BK}
\begin{proposition}\label{prop:2.10}
If $(V,\vac,Y)$ is a field algebra,
then $(V,\vac,Y^{op})$ is a field algebra as well.
\end{proposition}

%%%%%%%%%%%%%%%%%%%%%%%%%%%%%%%%%%%%
\subsection{Vertex algebras}\label{sec:2.3}

Vertex algebras were introduced by Borcherds in \cite{B} as a pointed vector space with $n$-products $a_{(n)} b$ for each $n \in \mb Z$, such that $\vac_{(n)} a = \delta_{n,-1} a$ and a (cubic) relation called Borcherds Identity holds. The following definition was given in \cite{K} where it was proved that it is equivalent to the Borcherds' one.

\begin{definition}(\cite{K})
\label{VADefinition}
A \emph{vertex algebra} 
is a pointed vector space $(V,\vac)$
with a state-field correspondence $Y:\,V\otimes V\to V((z))$
(cf. Definition \ref{pointedVectorSpaceWithStateFieldCorrespondenceDef})
such that each pair of quantum fields $(Y(a,z),Y(b,z))$
($a,b\in V$) is local (cf. Definition \ref{def:local}), i.e.
\begin{equation}
\label{locality}
(z-w)^NY(z) \big( 1 \otimes Y(w) \big)(a \otimes b \otimes c)
=(z-w)^N Y(w) \big( 1 \otimes Y(z) \big)(b \otimes a \otimes c)\,,
\end{equation}
for some $N=N(a,b) \in \mb Z_{\geq 0}$.
\end{definition}
\begin{remark}
In fact, in the definition of a vertex algebra
the translation covariance 2 axiom is redundant.
Indeed,
as we will point out in Remark \ref{YTspdY},
it follows from the associativity relation \eqref{fAssociativity},
which in turn follows from the translation covariance 1
axiom and the locality axiom \eqref{locality} (cf. Proposition \ref{VAsfAssociativity}).
\end{remark}

\begin{example}
Let $\mf{g}$ be a finite dimensional Lie algebra 
with a non-degenerate symmetric invariant bilinear form $(\cdot |\cdot )$. 
Let $\widehat{\mf{g}} = \mf{g}\otimes \mb{K}((t)) \oplus \mb{K} K$ be the centrally extended loop algebra with the commutation relations given by
\begin{equation}
\big[ a \otimes f(t), b \otimes g(t) \big]
= [a,b] \otimes f(t)g(t) - (a|b)K Res_{t} (f(t) g'(t))
\,,
\end{equation}
for $a,b\in\mf g$ and $f(t),g(t)\in\mb K((t)$,
and $K$ is central.
For $k \in \mb{K}$,
consider the vacuum $\widehat{\mf{g}}$-module of level $k$:
\begin{displaymath}
V^k (\mf{g}) = U(\widehat{\mf{g}}) \otimes_{U( \mf{g}[[t]] \oplus \mb{K} K )} \mb{K}_{k}
\,,
\end{displaymath}
where $\mb{K}_k$ is the one-dimensional representation of $\mf{g}[[t]] \oplus \mb{K} K$
on which $K=k$ and $\mf{g}[[t]]=0$.
On $V=V^k (\mf{g})$ we define a vertex algebra structure  as follows.
First, 
for $x \in \mf{g}$, we let
\begin{equation}
x(z) = \sum_{i \in \mathbb{Z}} \big( x \otimes t^i \big) z^{-i-1}
\,\in\End V[[z,z^{-1}]]\,.
\end{equation}
We also let $x_+(z)=\sum_{i<0} \big( x \otimes t^i \big) z^{-i-1}$
and $x_-(z) = \sum_{i\geq0} \big( x \otimes t^i \big) z^{-i-1}$.
Clearly, the vector space $V^k(\mf g)$ is spanned by the coefficients
the series 
$$
x^1_+(u_1) \cdots x^m_+(u_m)\vac
\,\in V((u_1))\dots((u_m))
\,,
$$
where $\vac$ is the image of $1$ in $V^k(\mf g)$.
For $x^1,x^2\in\mf g$, we define the normally ordered product:
$$
:x^1(z)x^2(z):\,\, = x^1_+(z) x^2(z) + x^2(z) x^1_- (z)
\,,
$$
and inductively, for $x^1,\dots,x^m\in\mf g$,
$$
: x^1(z) \cdots x^m(z) :\,\, = x^1_+(z) : x^2(z) \cdots x^m(z) : + : x^2(z) \cdots x^m(z) : x^1_- (z)
\,.
$$
We then define the translation operator $T\in\End V$ by
$$
e^{zT} x^1_+(u_1) \cdots x^m_+(u_m)\vac = x^1_+(z+ u_1) \cdots x^m_+(z+ u_m)\vac
\,,
$$
and the state-field correspondence $Y$ by
$$
Y\big( x^1_+(u_1) \cdots x^m_+(u_m)\vac, z \big) = \gse{z}{u_1} \cdots \gse{z}{u_m} : x^1(z+ u_1) \cdots x^m(z+ u_m):
\,.
$$
The resulting vertex algebra $V^k (\mf{g})$ is known as the \emph{level $k$ universal affine vertex algebra} 
of $\mf g$.
The fact that this is a vertex algebra follows from the general extension theorem of \cite{DSK06} and Taylor's formula: $\gse{z}{u} x(z+u) = \sum_{j \geq 0} \frac{u^j}{j!}\spd{z}^j\, x(z)$.
\end{example}

%%%%%%%%%%%%%%%%%%%%%%%%%%%%%%%%%%%%
\subsection{Associativity and other properties}

In this Section we recall some results from \cite{G},  \cite{K}, \cite{BK}.
\begin{lemma}[Goddard's uniqueness Theorem, \cite{G}, {\cite[Thm.4.4]{K}}]
\label{Goddard}
Let $V$ be a vertex algebra and let $a \in V$ and $a(z) \in \End V[[z, z^{-1}]]$ 
be an $\End V$-valued quantum field such that
$a(z) \vac = e^{zT}a$ %Y(z)(a \otimes \vac)
and $a(z)$ is local with $Y(z)(c \otimes -)$ for all $c \in V$. Then
\begin{equation}
a(z)b = Y(z)(a \otimes b)\ \textrm{for all}\ b \in V.
\end{equation}
\end{lemma}

\begin{remark}
\label{YTspdY}
In the proof of Goddard's uniqueness Theorem one does not need 
the translation covariance 2 axiom
of Definition \ref{pointedVectorSpaceWithStateFieldCorrespondenceDef}.
On the other hand, the translation covariance 2 axiom
can be derived using Goddard's uniqueness Theorem
and the other vertex algebra axioms as follows.
By Proposition \ref{BKProp.2.7} (a), we have
\begin{displaymath}
\spd{z} Y(z)(a \otimes \vac) = Y(z)(Ta \otimes \vac)
\,.
\end{displaymath}
Moreover, by the locality axiom, we have
\begin{displaymath}
(z-w)^{N+1} \spd{z}Y(z) \big( 1 \otimes Y(w) \big) (a \otimes c \otimes -) = (z-w)^{N+1} Y(w) \big( 1 \otimes \spd{z}Y(z) \big) (c \otimes a \otimes -).
\end{displaymath}
Therefore $\spd{z}Y(z)(a \otimes -)$ is local with $Y(w)(c \otimes -)$ for any $c \in V$. 
Hence, by Goddard's uniqueness Theorem we conclude that 
$\spd{z}Y(z)(a \otimes -)=Y(z)(Ta \otimes -)$.
\end{remark}

The proof of the following Lemma is the same as the proof 
in \cite{FLM} or \cite[Prop.4.2]{K}.
\begin{lemma}[Skew-symmetry]
\label{VAsSkewSymmLem}
Let $Y$ be a state-field correspondence on a pointed vector space $(V,\vac)$.
Let $a,b\in V$.
Assuming that
$Y(z)(a \otimes -)$ and $Y(z)(b \otimes -)$ are local on $\vac$, 
we have
\begin{equation}
\label{VAsSkewSymm}
Y(z)(a \otimes b) = e^{zT}Y(-z)(b \otimes a)
\,\,\Big(=Y^{op}(z)(a\otimes b)\Big)
\,.
\end{equation}
\end{lemma}

\begin{proposition}[cf. {\cite[Prop.4.1]{BK}}]
\label{VAsfAssociativity}
If $V$ is a vertex algebra, then the state-field correspondence $Y$ 
satisfies the associativity relation \eqref{fAssociativity}.
\end{proposition}
\begin{remark}
\label{GenVAsfAssociativity}
In the statement of Proposition \ref{VAsfAssociativity}
one can weaken the assumptions by requiring that
$Y$ is a state-field correspondence satisfying the locality on any vector of $V$.
\end{remark}

It follows, by Proposition \ref{VAsfAssociativity} and Lemma \ref{VAsSkewSymmLem}, 
that a vertex algebra is a field algebra satisfying $Y = Y^{op}$. 
The converse is also true:
\begin{theorem}[{\cite[Thm.7.3]{BK}}]
\label{BKThm.7.3}
A vertex algebra is the same as a field algebra for which $Y = Y^{op}$.
\end{theorem}

\begin{corollary}[{\cite{BK}}]
\label{EquivalenceLocalityOnAnyVectorLocality}
Let $(V, \vac)$ be a pointed vector space with a state-field correspondence $Y$. Then $Y$ is local if and only if it is local on every vector of $V$.
\end{corollary}

%%%%%%%%%%%%%%%%%%%%%%%%%%%%%%%%%%%%
\subsection{Equivalent characterizations of vertex algebras}

Let $Y$ be a state-field correspondence on the pointed vector space $(V, \vac)$
(cf. Definition \ref{pointedVectorSpaceWithStateFieldCorrespondenceDef}).

The \emph{$n$-product} of the quantum fields $Y(a,z)$ and $Y(b,z)$ is
defined as ($a,b\in V$, $n\in\mb Z$):
\begin{equation}
\label{nproducts}
Y(a,z)_{(n)} Y(b,z) 
= 
Res_x \Big( Y(a,x) Y(b,z) \gse{x}{z}(x-z)^n
- Y(b,z) Y(a,x) \gse{z}{x}(x-z)^n \Big)
\,.
\end{equation}
\begin{remark}
As an immediate consequence of Theorem \ref{KDecThm},
in a vertex algebra we have
\begin{equation}
[Y(a,z),Y(b,w)] = \sum_{j \geq 0} (Y(a,z)_{(j)} Y(b,z)) \frac{\partial^{j}_w \delta(z,w)}{j!}.
\end{equation}
\end{remark}

The \emph{Jacobi identity} on $V$  (cf. \cite[Sec.8.8]{FLM}) is
($a,b,c\in V$):
\begin{equation}
\label{Jacobi}
\begin{split}
\gse{z}{w} \delta(x, z-w) &Y(a,z) Y(b,w) c
- \gse{w}{z} \delta(x, z-w) Y(b,w) Y(a,z) c \\
&= 
\gse{z}{x} \delta(w, z-x) Y(Y(a,x)b,w)c
\,.
\end{split}
\end{equation}

The \emph{Borcherds identities} on $V$ \cite{BK} is ($a, b, c \in V$, $n \in \mb Z$):
\begin{equation}
\label{BorcherdsIdentities}
\begin{split}
Y(a,z) & Y(b,w) c\ \gse{z}{w}(z-w)^n
-Y(b,w) Y(a,z) c\ \gse{w}{z}(z-w)^n \\
&= \sum_{j \geq 0} Y(a_{(n+j)} b,w) c\ \frac{\spd{w}^j \delta(z,w)}{j!}.
\end{split}
\end{equation}

The \emph{$n$-product identity} on $V$ \cite{BK} is ($a, b \in V$, $n \in \mb Z$):
\begin{equation}
\label{nproductIdentities}
Y(a,z)_{(n)} Y(b,z) = Y(a_{(n)} b,z)
\,.
\end{equation}

\begin{theorem}
\label{CharacterizationVAs}
Let $Y$ be a state-field correspondence on the pointed vector space $(V, \vac)$.
The following statements are equivalent:
\begin{enumerate}[(i)]
\item \label{VVA} $V$ is a vertex algebra, i.e., the locality axiom \eqref{locality} holds;
\item \label{VJac} the Jacobi identity \eqref{Jacobi} holds;
\item \label{VFAYYop} $V$ is a field algebra
(i.e., the associativity relation \eqref{fAssociativity} holds), 
and $Y = Y^{op}$;
\item \label{VBI} the Borcherds identities \eqref{BorcherdsIdentities} hold;
\item \label{Vnploc} the $n$-product identity \eqref{nproductIdentities}
and the locality axiom \eqref{locality} hold; 
\item \label{VnpYYop} the $n$-product identity \eqref{nproductIdentities} holds 
and $Y = Y^{op}$.
\end{enumerate}
\end{theorem}

\begin{proof}
The equivalence of \eqref{VVA} and \eqref{VFAYYop} is Theorem \ref{BKThm.7.3}.
The equivalence of \eqref{VVA} and \eqref{VJac} was proved in \cite[Thm.3.6.3]{LL}.
However, the proof is easily obtained as follows:
applying Lemma \ref{LiLem.2.1} to $a(z,w)=Y(a,z)Y(b,w)$,
$b(z,w)=Y(b,w) Y(a,z) c$ and $c(x,w)=Y(Y(a,x)b,w)c$,
we immediately get that the Jacobi identity \eqref{Jacobi}
is equivalent to the locality of $(Y(a,z),Y(b,z))$ on any vector $c\in V$
and the associativity relation \eqref{fAssociativity}.
In particular \eqref{VVA} implies \eqref{VJac}.
On the other hand, the locality on the vacuum vector implies $Y=Y^{op}$.
Hence, \eqref{VJac} implies \eqref{VFAYYop}.

Next, we prove that \eqref{VJac} implies \eqref{VBI}.
Multiplying both sides of the Jacobi identity \eqref{Jacobi} by $x^n$ and taking the residue $Res_x$ one has
\begin{align*}
Y(a,z) & Y(b,w) c\ \gse{z}{w}(z-w)^n
-Y(b,w) Y(a,z) c\ \gse{w}{z}(z-w)^n \\
& = 
Y\Big(
Res_x \big( x^n \gse{z}{x} \delta(w, z-x) Y(a,x) b \big)
,w\Big)c\,.
\end{align*}
On the other hand the following equalities hold:
\begin{align*}
&Res_x \big( x^n \gse{z}{x} \delta(w, z-x) Y(a,x) b \big)
=Res_x \big( x^n e^{-x \spd{z}} \delta(w, z) Y(a,x) b \big)\\
&= \sum_{j \geq 0} (-1)^j (a_{(n+j)} b) \frac{\spd{z}^j}{j!} \delta(z,w)
= \sum_{j \geq 0} (a_{(n+j)} b) \frac{\spd{w}^j}{j!} \delta(z,w).
\end{align*}
Combining the above two equations we get the Borcherds identity 
\eqref{BorcherdsIdentities}, proving \eqref{VBI}.

The Borcherds identity \eqref{BorcherdsIdentities}
immediately implies both the
$n$-th product identity \eqref{nproductIdentities},
by taking the residue in $z$,
and the locality axiom \eqref{locality},
since for $n\gg0$ we have $a_{(n)}b=0$.
Hence, \eqref{VBI} implies \eqref{Vnploc}.

Condition \eqref{Vnploc} obviously implies \eqref{VnpYYop}, 
by Lemma \ref{VAsSkewSymmLem}.
Finally \eqref{VnpYYop} implies \eqref{VVA} by the following
\begin{lemma}[{\cite[Thm.4.1, (a)]{BK}}]
\label{EquivalencenproductIdentitiesLemma}
Let $(V, \vac)$ be a pointed vector with a state-field correspondence $Y$. Then $Y$ satisfies the $n$-product identities if and only if
\begin{equation}
\label{EquivalencenproductIdentities}
\big[ Y(a,z), Y^{op}(b,w) \big] = \sum_{\substack{j \geq 0\\ finite}} Y^{op} \big( a_{(j)}b, w \big) \frac{\spd{w}^j}{j!} \delta(z,w)
\end{equation}
for any $a, b \in V$.
\end{lemma}
\end{proof}

%%%%%%%%%%%%%%%%%%%%%%%%%%%%%%%%
\subsection{Commutative vertex algebras}
\label{CommutativeVAsSection}

\begin{definition}\label{def:holomorphic}
A state-field correspondence $Y$ on a pointed vector space $(V, \vac)$
is \emph{holomorphic} if $Y(z)(a\otimes b)\in V[[z]]$ for all $a,b\in V$.
\end{definition}
\begin{lemma}
\label{vacUnit-1prod}
Let $Y$ be a state-field correspondence
on a pointed vector space $(V, \vac)$
(cf. Definition \ref{pointedVectorSpaceWithStateFieldCorrespondenceDef})
and consider the $(-1)$-product defined by $a_{(-1)}b=\res_zz^{-1}Y(z)(a\otimes b)$.
Then the vacuum vector $\vac \in V$ is the unit element of the $(-1)$-product 
and $T$ is a derivation of the $(-1)$-product.
\end{lemma}

\begin{proof}
It is a straightforward consequence of the vacuum and translation covariance axioms.
\end{proof}

\begin{lemma}
\label{fAssociativity-1prodAssociativity}
If $V$ is a field algebra with a holomorphic state-field correspondence $Y$,
then the $(-1)$-product is associative.
\end{lemma}
\begin{proof}
Since the algebra $\mb K[[z,w]]$ of formal power series in $z,w$ has no zero divisors,
the associativity relation \eqref{fAssociativity} implies
$$
Y(w) \big( 1 \otimes Y(z) \big)( a \otimes b \otimes c) 
= Y(z+w) \big( 1 \otimes Y(w) \big) (a \otimes b \otimes c),
$$
in $V[[z,w]]$.
Multiplying both sides of the above equation
by $z^{-1}w^{-1}$ and taking the residues in $z$ and $w$,
we obtain the claim.
\end{proof}

\begin{lemma}
\label{LocalityOnVacuum-1prodCommutativity}
If $Y$ is a holomorphic state-field correspondence 
such that $Y=Y^{op}$,
then the $(-1)$-product is commutative.
\end{lemma}
\begin{proof}
The claim is obtained 
multiplying both sides of the equation $Y(z)(a\otimes b)=Y^{op}(z)(a\otimes b)$
by $z^{-1}$ and taking the residue in $z$.
\end{proof}

Recall that 
a vertex algebra $V$ is said to be \emph{commutative} if
the locality axiom \eqref{locality} holds with $N=0$
for all $a,b\in V$.
\begin{theorem}[{\cite{B},\cite[Sec.1.4]{K}}]\label{CommVAHolom}
A vertex algebra $V$ is commutative if and only if the state-field correspondence $Y$
is holomorphic.
In this case,
the $(-1)$-product is commutative, associative, unital, with derivation $T$,
and the state-field correspondence $Y$ is given by:
$Y(z) (a \otimes b) = (e^{zT}a)_{(-1)} b$. Thus, commutative vertex algebras are in bijective correspondence with unital commutative associative differential algebras.
\end{theorem}
\begin{proof}
First, assume that the vertex algebra $V$ is commutative.
Letting $c=\vac$ in the locality axiom \eqref{locality} (with $N=0$)
and using Proposition \ref{BKProp.2.7}(a), we get
\begin{equation}
\label{CommVAHolomEq1}
Y(z) \big(1 \otimes e^{wT}\big) (a \otimes b) = Y(w) \big( 1 \otimes e^{zT})(b \otimes a).
\end{equation}
Multiplying both sides of equation \eqref{CommVAHolomEq1} by $w^{-1}$ 
and taking the residue $Res_w$, one has
\begin{equation}\label{damettere}
Y(z)(a \otimes b) = b_{(-1)} e^{zT} a\,\in V[[z]].
\end{equation}
Hence, $Y$ is holomorphic.
Conversely,
if $Y$ is holomorphic, the locality equation \eqref{locality}
can be divided by $(z-w)^N$ since $\mb K[[z,w]]$ has no zero divisors.
This proves the first assertion of the theorem.
The remaining assertions follow from Lemmas \ref{vacUnit-1prod},  \ref{fAssociativity-1prodAssociativity}, \ref{LocalityOnVacuum-1prodCommutativity},
and equation \eqref{damettere}.
\end{proof}

%%%%%%%%%%%%%%%%%%%%%%%%%%%%%%%%%%
\section{Quantum vertex algebras}\label{sec:3}

\subsection{Topologically free $\Kh$-modules}
\label{TopologicallyFreeKhModulesSection}

Throughout the rest of the paper we shall work over the algebra $\mb K[[h]]$
of formal power series in the variable $h$,
and all the algebraic structures that we will consider
are modules over $\mb K[[h]]$.
\begin{definition}
\label{KasTopologicallyFreeModule}
A \emph{topologically free} $\mb K[[h]]$-module 
is isomorphic to $W[[h]]$ for some $\mb K$-vector space $W$.
%The \emph{rank} of a topologically free $\mb K[[h]]$-module.
%is defined as $\rk(W[[h]]) = \dim W$.
\end{definition}
Note that $W[[h]]\not\cong W\otimes\mb K[[h]]$,
unless $W$ is finite-dimensional over $\mb K$,
and that the tensor product 
$U[[h]]\otimes_{\mb K[[h]]} W[[h]]$
of topologically free $\mb K[[h]]$-modules
is not topologically free,
unless one of $U$ and $W$ is finite dimensional.
One defines, for any vector spaces $U$ and $W$, the \emph{completed} tensor product by
\begin{equation}\label{eq:compl-tensor}
U[[h]]\widehat{\otimes}_{\mb K[[h]]}W[[h]]
:=
(U\otimes W)[[h]]
\,.
\end{equation}
This is a completion in $h$-adic topology of $U[[h]]\otimes_{\mb K[[h]]} W[[h]]$
\cite{Kas}.

%Throughout the rest of the paper,
%we will omit the completion sign,
%and we will denote the completed tensor product of topologically free
%$\mb K[[h]]$-modules by $\otimes$.

Given a topologically free $\mb K[[h]]$-module $V$,
we let
\begin{equation}\label{eq:Vh}
V_h((z)) = 
\Big\{
a(z) \in V[[z, z^{-1}]] 
\,\Big|\,
a(z) \in V((z)) \mod h^M 
\text{ for every }
M \in \mb Z_{\geq0}
\Big\}
\,.
\end{equation}
In other words, expanding 
$a(z)=\sum_{n \in \mathbb{Z}} a_{(n)} z^{-n-1}$,
we require that $\lim_{n\to+\infty}a_{(n)}=0$ in $h$-adic topology.

%%%
\subsection{Quantum fields}
\label{sec:3.25}

Let $V$ be a topologically free $\Kh$-module.
An $\End_{\Kh}V$-\emph{valued quantum field} is
an $\End_{\Kh}V$-valued formal distribution $a(z)$
such that $a(z)b \in V_h((z))$ for any $b \in V$.

The following two simple results generalize Proposition \ref{BKProp.2.7}
for quantum fields over $\mb K[[h]]$. Proof is the same, we reprocude it for the sake of completeness.
\begin{lemma}
\label{QK15Lem.1.2}
%Let $V$ a topologically free $\mb{K}[[h]]$-module, 
Let $\vac \in V$ and $T: V \rightarrow V$ be a $\mb K[[h]]$-linear map such that $T\vac = 0$. 
Then for any $\End_{\Kh}V$-valued quantum field $a(z)$
such that $[T,a(z)]=\spd z a(z)$ (translation covariance), 
we have
\begin{equation}
a(z)\vac = e^{zT}a = \sum_{k \geq 0}\frac{T^k a}{k!}z^k,
\end{equation}
where $a=\Res_z z^{-1}a(z)\vac$.
\end{lemma}
\begin{proof}
Applying both sides of the translation covariance assumption to the vacuum vector
we get 
\begin{equation}\label{eq:last}
Ta(z)\vac=\spd z a(z)\vac
\,.
\end{equation}
Since $a(z)\vac$ is a Laurent series in $z$, 
it follows from \eqref{eq:last} that it is actually a formal power series in $z$,
and the claim immediately follows.
\end{proof}
\begin{lemma}
\label{QK15Lem.1.2b}
Let $T: V \rightarrow V$ be a $\mb K[[h]]$-linear map
and let $a(z)$ be an $\End_{\Kh}V$-valued quantum field 
such that $[T,a(z)]=\spd z a(z)$.
We have
\begin{equation}
e^{wT}a(z)e^{-wT}=\iota_{z,w}a(z+w)
\,.
\end{equation}
\end{lemma}
\begin{proof}
This is an immediate consequence of the Taylor expansion 
and the translation covariance assumption.
\end{proof}

%%%
\subsection{Braided vertex algebras}
\label{sec:3.2}

In this subsection we propose a definition of a braided vertex algebra,
which is slightly less restrictive than that in \cite{EK5}.

\begin{definition}
\label{def:state-field}
Let $V$ be a topologically free $\mb{K}[[h]]$-module,
with a given non-zero vector $\vac \in V$ (vacuum vector),
and a $\mb K[[h]]$-linear map $T:\,V\to V$
such that $T(\vac)=0$ (translation operator).
\begin{enumerate}[(a)]
\item
A \emph{topological state-field correspondence} on $V$ is a $\mb K[[h]]$-linear map 
\begin{equation}\label{eq:Y}
Y: V\widehat{\otimes} V \rightarrow V_h((z))
\,,
\end{equation}
satisfying the axioms of a state-field correspondence as in Definition \ref{pointedVectorSpaceWithStateFieldCorrespondenceDef}:
\begin{enumerate}[(i)]
\item 
$Y(z)(\vac\otimes a) = a$,
and $Y(z)(a\otimes \vac) \in a+ V[[z]]z$,
for all $a\in V$ (vacuum axioms);
\item 
$\spd{z}Y(z) = TY(z)-Y(z)(1\otimes T) = Y(z)(T\otimes1)$ 
(translation covariance).
\end{enumerate}
\item
A \emph{braiding} on $V$ is 
a $\mb K[[h]]$-linear map 
\begin{equation}\label{eq:S}
\mc{S}: V \widehat{\otimes} V \rightarrow V\widehat{\otimes} V \widehat{\otimes} \big(\mb{K}((z))[[h]]\big)
\,,
\end{equation}
such that $\mc{S} = 1 + O(h)$.
\end{enumerate}
\end{definition}

\begin{definition}
\label{BVADef}
A \emph{braided vertex algebra} is a quintuple $(V,\vac,T,Y,\mc S)$ as in Definition \ref{def:state-field},
satisfying the following $\mc{S}$-\emph{locality}:
for every $a, b \in V$ and $M \in \mb Z_{\geq0}$, there exists $N = N(a,b, M) \geq 0$ such that
\begin{equation}
\begin{split}\label{S-locality}
&(z-w)^N Y(z) \big( 1\otimes Y(w) \big)\big(\mc{S}(z-w)(a \otimes b) \otimes c \big)\\
&=(z-w)^N Y(w) \big( 1\otimes Y(z) \big)(b \otimes a \otimes c) \quad \textrm{mod}\ h^M
\end{split}
\end{equation}
for all $c\in V$.
\end{definition}
Note that for a braided vertex algebra $V$, the vector space $V/hV$
carries a canonical structure of a vertex algebra.

As for the usual state-field correspondence,
for the topological state-field correspondence 
$T$ is determined by the map $Y$ and the vacuum vector $\vac$:
\begin{equation}
T(a) = a_{(-2)}\vac = Res_z \big( z^{-2} Y(z) (a \otimes \vac) \big).
\end{equation}
Thus, a braided vertex algebra is a quadruple $(V, \vac, Y, \mc S)$ satisfying the axioms 
of Definitions \ref{def:state-field} and \ref{BVADef}.
By abuse of terminology, we shall call the quadruple $(V, \vac, Y, \mc S)$ satisfying the axioms 
of Definitions \ref{def:state-field}
a \emph{braided state-field correspondence}.

\begin{lemma}[{\cite[Lem.1.2]{EK5}}]
\label{EK5Lem.1.2}
Recall the definition \eqref{eq:Yop} of $Y^{op}$.
In a braided vertex algebra $V$, we have
\begin{equation}
\label{EK5Lem.1.2Eq}
Y(z)\mc{S}(z)(a\otimes b) = Y^{op}(z)(a\otimes b)
\,\,\,\text{ for all } a,b\in V
\,.
\end{equation}
\end{lemma}

\begin{proof}
By the $\mc{S}$-locality \eqref{S-locality} with $c = \vac$ and Lemma \ref{QK15Lem.1.2}, 
we have that, for any $a, b \in V$ and $M \in \mb Z_{\geq0}$, 
there exists $N = N(a, b, M) \geq 0$ such that
\begin{equation}
\label{EK5Lem.1.2Eq1}
(z-w)^N Y(z) \big( 1 \otimes e^{wT} \big) \mc{S}(z-w)(a \otimes b) = (z-w)^N Y(w) \big( b \otimes e^{zT}a\big)\ \textrm{mod}\ h^M.
\end{equation}
By Lemma \ref{QK15Lem.1.2b}, equation \eqref{EK5Lem.1.2Eq1} becomes
\begin{equation}
\label{EK5Lem.1.2Eq2}
(z-w)^N e^{wT}\gse{z}{w}Y(z-w) \mc{S}(z-w)(a \otimes b)
= 
(z-w)^N Y(w) \big( b \otimes e^{zT}a\big)\ \textrm{mod}\ h^M.
\end{equation}
For $N$ big enough, $(z-w)^N Y(z-w) \mc{S}(z-w)(a \otimes b) \in V[[z-w]]$ mod $h^M$. Therefore both sides of equation \eqref{EK5Lem.1.2Eq2} have only positive powers of $z$. Evaluating equation \eqref{EK5Lem.1.2Eq2} on $z=0$, we obtain
\begin{equation}
\label{EK5Lem.1.2Eq3}
(-w)^N e^{wT}Y(-w) \mc{S}(-w)(a \otimes b) = (-w)^N Y(w) ( b \otimes a)\ \textrm{mod}\ h^M.
\end{equation}
Multiplying both sides of equation \eqref{EK5Lem.1.2Eq3} by $(-w)^{-N} e^{-wT}$ and renaming $-w$ in $z$, one has
\begin{displaymath}
Y(z) \mc{S}(z)(a \otimes b) = e^{zT} Y(-z) ( b \otimes a)\ \textrm{mod}\ h^M.
\end{displaymath}
Since the above equation holds modulo $h^M$ for every $M$,
it must hold identically.
The claim follows.
\end{proof}

\begin{remark}\label{EK5Lem.1.2a}
The same proof as for Lemma \ref{EK5Lem.1.2} shows that,
if the $\mc S$-locality \eqref{S-locality} holds for $c=\vac$,
then $Y\mc S=Y^{op}$.
\end{remark}

For braided vertex algebras, an analogue of the associativity relation \eqref{fAssociativity} does not hold, 
but one has the following quasi-associativity relation \eqref{hfQuasiAssociativity}:
\begin{proposition}[{\cite[Prop.1.1]{EK5}}]
\label{EK5Prop.1.1}
Let $V$ be a braided vertex algebra.
For every $a, b, c \in V$ and $M \in \mb Z_{\geq0}$, there exists $N \geq 0$ such that
\begin{equation}
\label{hfQuasiAssociativity}
\begin{split}
\gse{z}{w} \big( &(z+w)^N Y(z+w) \big( 1\otimes Y(w) \big)\Sp{2}{3}(w)\Sp{1}{3}(z+w) (a \otimes b \otimes c) \big)\\
&= (z+w)^N Y(w)\mc{S}(w) \big( Y(z) \otimes 1 \big) (a \otimes b \otimes c)\ \textrm{mod}\ h^M.
\end{split}
\end{equation}
\end{proposition}
\begin{proof}
Applying $e^{-wT}$ to both sides of the $\mc S$-locality \eqref{S-locality}, one has
\begin{equation}
\label{eq:EK5Prop.1.1}
\begin{split}
&(z-w)^N e^{-wT} Y(z) \big( 1\otimes Y(w) \big)\big(\mc{S}(z-w)(a \otimes b) \otimes c \big)\\
&=(z-w)^N e^{-wT} Y(w) \big( 1\otimes Y(z) \big)(b \otimes a \otimes c) \quad \textrm{mod}\ h^M.
\end{split}
\end{equation}
By Lemmas \ref{QK15Lem.1.2b} and \ref{EK5Lem.1.2}, the left hand side of equation \eqref{eq:EK5Prop.1.1} is equal to the following:
\begin{equation}\label{eq:proof1}
\begin{split}
&(z-w)^N \iota_{z,w} Y(z-w) \big( 1\otimes e^{-wT} Y(w) \big)\big(\mc{S}(z-w)(a \otimes b) \otimes c \big)\\
&= (z-w)^N \iota_{z,w} Y(z-w) \big( 1\otimes Y(-w) \big) \mc{S}^{2 3}(-w) \mc{S}^{1 3}(z-w) (a \otimes c \otimes b)\\
&= \iota_{z,w} \left((z-w)^N Y(z-w) \big( 1\otimes Y(-w) \big) \mc{S}^{2 3}(-w) \mc{S}^{1 3}(z-w) (a \otimes c \otimes b) \right).
\end{split}
\end{equation}
Similarly, using Lemma \ref{QK15Lem.1.2b} on the right hand side of equation \eqref{eq:EK5Prop.1.1}, 
one has
\begin{equation}\label{eq:proof2}
(z-w)^N Y(-w) \mc S(-w) \big( Y(z) \otimes 1 \big)(a \otimes c \otimes b).
\end{equation}
Equation \eqref{hfQuasiAssociativity} is obtained by equating (\!\!\!$\mod h^M$) 
\eqref{eq:proof1} and \eqref{eq:proof2},
changing the sign of $w$ and switching the letters $b$ and $c$.
\end{proof}

%\begin{definition}\label{def:non-degenerate}
%A braided vertex algebra $V$ is called \emph{non-degenerate}
%if the topological state-field correspondence $Y$ in \eqref{eq:Y} is injective.
%\end{definition}
%%
%This definition is similar to but different from the non-degeneracy condition in \cite[Def.1.9.1]{EK5}.
%Note that if the vertex algebra $V/hV$ is non-degenerate, i.e. its state-field correspondence
%is injective, then the braided vertex algebra $V$ is automatically non-degenerate.
%In particular, this holds when $V/hV$ is a freely generated vertex algebra.

We have used above the following standard notation:
given $n\geq2$ and $i,j\in\{1,\dots,n\}$, we let
\begin{equation}\label{eq:Sij}
\mc S^{ij}(z)\,:\,\,
V^{\widehat{\otimes} n}\,\longrightarrow\,
V^{\widehat{\otimes} n}\widehat{\otimes}(\mb K((z))[[h]])
\,,
\end{equation}
act on the $i$-th and $j$-th factors (in this order) of $V^{\widehat{\otimes}n}$,
leaving the other factors unchanged.
For example, for $n=2$ we have $\mc S^{21}(z)=(12)\circ\mc S(z)\circ(12)$,
where $(12)$ is the transposition of factors in $V^{\widehat{\otimes}2}$.

\begin{proposition}\label{prop:non-deg}
Let $(V,\vac,T,Y,\mc S)$ be a braided vertex algebra.
Extend $Y(z)$ to a map 
$V\widehat{\otimes} V\widehat{\otimes}\big(\mb K((z))[[h]]\big)\to V_h((z))$
in the obvious way.
Then, modulo $\ker Y(z)$, we have
\begin{enumerate}[(a)]
\item
$\mc S(z)(\vac\otimes a)\equiv\vac\otimes a$, and $\mc S(z)(a\otimes\vac)\equiv a\otimes\vac$;
\item
$[T\otimes 1, \mc{S}(z)] \equiv -\spd{z}\mc{S}(z)$ 
(left shift condition);
\item
$[1\otimes T, \mc{S}(z)] \equiv \spd{z}\mc{S}(z)$ 
(right shift condition);
\item
$[T\otimes1+1\otimes T, \mc{S}(z)] \equiv 0$;
\item
$\mc S(z)\mc S^{21}(-z)=1$ (unitarity).
\end{enumerate}
Moreover, we have the quantum Yang-Baxter equation:
\begin{enumerate}[(a)]
\setcounter{enumi}{5}
\item
$\Sp{1}{2}(z_1-z_2)\Sp{1}{3}(z_1-z_3)\Sp{2}{3}(z_2-z_3) 
\equiv \Sp{2}{3}(z_2-z_3)\Sp{1}{3}(z_2-z_3)\Sp{1}{2}(z_1-z_2)$,
modulo $\ker(Y(z_1)(1\otimes Y(z_2))(1^{\otimes2}\otimes Y(z_3)(-\otimes-\otimes-\otimes\vac))$.
\end{enumerate}
\end{proposition}
\begin{proof}
As for the usual state-field correspondence,
for a topological state-field correspondence $Y$
the map $Y^{op}$ defined by \eqref{eq:Yop}
is a topological state-field correspondence as well (cf. Proposition \ref{BKProp.2.8}).
Moreover, by Lemma \ref{EK5Lem.1.2}, $Y^{op}=Y\mc S$.

By the equation $Y\mc S = Y^{op}$, the definition \eqref{eq:Yop} of $Y^{op}$,
the vacuum axiom and translation covariance, to get
$$
Y(z) \mc S(z) (a \otimes \vac) = Y^{op}(z) (a \otimes \vac) 
= e^{zT} Y(-z) (\vac \otimes a) = e^{zT} a = Y(z) (a \otimes \vac)
\,,
$$
and
$$
Y(z) \mc S(z) (\vac \otimes a) = Y^{op}(z) (\vac \otimes a) 
= e^{zT} Y(-z) (a \otimes \vac) %= e^{zT} e^{-zT} a 
= a
= Y(z) (\vac \otimes a)
\,,
$$
proving claim (a).

Since $T$ is a derivation of $Y^{op}$, we have
\begin{equation}\label{eq:iiia}
T Y(z) \mc S(z) (a \otimes b) = Y(z) \mc S(z) (T \otimes 1)(a \otimes b) 
+ Y(z) \mc S(z) (1 \otimes T)(a \otimes b)
\,,
\end{equation}
and since $T$ is a derivation of $Y$, we also have
\begin{equation}\label{eq:iiib}
T Y(z) \mc S(z) (a \otimes b) 
= Y(z) (T \otimes 1) \mc S(z) (a \otimes b) + Y(z) (1 \otimes T) \mc S(z) (a \otimes b)
\,.
\end{equation}
Combining equations \eqref{eq:iiia} and \eqref{eq:iiib}, we get 
$$
Y(z) [T \otimes 1+1\otimes T, \mc S(z)](a \otimes b)
= 0 \,,
$$
proving claim (d).
By the translation covariance of $Y^{op}$, we have
\begin{equation}\label{eq:iva}
T Y(z) \mc S(z) (a \otimes b) 
= 
Y(z) \mc S(z) (1 \otimes T)(a \otimes b) 
+ \spd{z} \left(Y(z) \mc S(z)\right)(a \otimes b)
\,,
\end{equation}
while, by the translation covariance of $Y$ we have
\begin{equation}\label{eq:ivb}
T Y(z) \mc S(z) (a \otimes b) 
= 
Y(z) (1 \otimes T) \mc S(z) (a \otimes b) 
+ \spd{z} (Y(z))\, \mc S(z)(a \otimes b)
\,.
\end{equation}
Combining equations \eqref{eq:iva} and \eqref{eq:ivb}, we get 
$$
Y(z) \left( [1 \otimes T, \mc S(z)] - \spd{z} \mc S(z)\right)(a \otimes b)=0
\,,
$$
proving claim (c).
Claim (b) is an obvious consequence of (c) and (d).

Next, let us prove the unitarity condition (e).
Since $Y(z)\mc S(z) = Y^{op}(z)$ and $\mc S^{2 1}(z) = (1\ 2)\mc S(z) (1\ 2)$, 
we have
\begin{align*}
&Y(z) \mc S(z) \mc S^{2 1}(-z) (a \otimes b)
%= Y(z) \mc S(z) (1\ 2) \mc S(-z) (b \otimes a) 
= Y^{op}(z) (1\ 2) \mc S(-z) (b \otimes a)\\
&= e^{zT} Y(-z) \mc S(-z) (b \otimes a) = e^{zT} Y^{op}(-z) (b \otimes a) = Y(z)(a \otimes b)
\,.
\end{align*}
For the last equality we used the obvious fact that $(Y^{op})^{op}=Y$.
Claim (e) follows.

Finally, we prove the quantum Yang-Baxter equation (f).
For $N_{1 2}, N_{1 3}, N_{2 3}$ large enough, using three times the $S$-locality, we have
\begin{align*}
& (z_1-z_2)^{N_{1 2}} (z_1-z_3)^{N_{1 3}} (z_2-z_3)^{N_{2 3}}
Y(z_1) (1 \otimes Y(z_2)) (1 \otimes 1 \otimes Y(z_3))  \cdot\\
&\cdot \mc S^{1 2}(z_1 - z_2) \mc S^{1 3}(z_1 - z_3) \mc S^{2 3}(z_2 - z_3) (a \otimes b \otimes c \otimes \vac)\\
%&= (z_1-z_2)^{N_{1 2}} (z_1-z_3)^{N_{1 3}} (z_2-z_3)^{N_{2 3}} Y(z_1) 
%(1 \otimes Y(z_2)) \mc S^{1 2}(z_1 - z_2) \cdot\\
%&\cdot (1 \otimes 1 \otimes Y(z_3)) \mc S^{1 3}(z_1 - z_3) \mc S^{2 3}(z_2 - z_3) 
%(a \otimes b \otimes c \otimes \vac)\\
&= (z_1-z_2)^{N_{1 2}} (z_1-z_3)^{N_{1 3}} (z_2-z_3)^{N_{2 3}} Y(z_2) (1 \otimes Y(z_1)) (1\ 2) \cdot\\
&\cdot (1 \otimes 1 \otimes Y(z_3)) \mc S^{1 3}(z_1 - z_3) \mc S^{2 3}(z_2 - z_3) (a \otimes b \otimes c \otimes \vac)\\
&= (z_1-z_2)^{N_{1 2}} (z_1-z_3)^{N_{1 3}} (z_2-z_3)^{N_{2 3}} Y(z_2) (1 \otimes Y(z_1)) (1 \otimes 1 \otimes Y(z_3)) \cdot\\
&\cdot \mc S^{2 3}(z_1 - z_3) \mc S^{1 3}(z_2 - z_3) (b \otimes a \otimes c \otimes \vac)\\
&= (z_1-z_2)^{N_{1 2}} (z_1-z_3)^{N_{1 3}} (z_2-z_3)^{N_{2 3}} Y(z_2) (1 \otimes Y(z_3)) (1 \otimes 1 \otimes Y(z_1)) \cdot\\
&\cdot (2\ 3) \mc S^{1 3}(z_2 - z_3) (b \otimes a \otimes c \otimes \vac)\\
&= (z_1-z_2)^{N_{1 2}} (z_1-z_3)^{N_{1 3}} (z_2-z_3)^{N_{2 3}} Y(z_2) (1 \otimes Y(z_3)) (1 \otimes 1 \otimes Y(z_1)) \cdot\\
&\cdot \mc S^{1 2}(z_2 - z_3) (b \otimes c \otimes a \otimes \vac)\\
%&= (z_1-z_2)^{N_{1 2}} (z_1-z_3)^{N_{1 3}} (z_2-z_3)^{N_{2 3}} Y(z_2) (1 \otimes Y(z_3)) 
%\mc S^{1 2}(z_2 - z_3) \cdot\\
%&\cdot (1 \otimes 1 \otimes Y(z_1)) (b \otimes c \otimes a \otimes \vac)\\
&= (z_1-z_2)^{N_{1 2}} (z_1-z_3)^{N_{1 3}} (z_2-z_3)^{N_{2 3}} Y(z_3) (1 \otimes Y(z_2)) (1 \otimes 1 \otimes Y(z_1)) \cdot\\
&\cdot  (c \otimes b \otimes a \otimes \vac)
\ \text{mod}\ h^M\,. 
%\\
%&= Y(z_3) (1 \otimes Y(z_2)) (1 \otimes 1 \otimes Y(z_1)) (z_1-z_2)^{N_{1 2}} (z_1-z_3)^{N_{1 3}} 
%(z_2-z_3)^{N_{2 3}} \cdot\\
%&\cdot  (c \otimes b \otimes a \otimes \vac)
\end{align*}
Similarly one gets
\begin{align*}
& (z_1-z_2)^{N_{1 2}} (z_1-z_3)^{N_{1 3}} (z_2-z_3)^{N_{2 3}}
Y(z_1) (1 \otimes Y(z_2)) (1 \otimes 1 \otimes Y(z_3))  \cdot\\
&\cdot \mc S^{2 3}(z_2 - z_3) \mc S^{1 3}(z_1 - z_3) \mc S^{1 2}(z_1 - z_2) (a \otimes b \otimes c \otimes \vac)\\
&= (z_1-z_2)^{N_{1 2}} (z_1-z_3)^{N_{1 3}} (z_2-z_3)^{N_{2 3}}
Y(z_3) (1 \otimes Y(z_2)) (1 \otimes 1 \otimes Y(z_1))  \cdot\\
&\cdot  (c \otimes b \otimes a \otimes \vac)\ \text{mod}\ h^M.
\end{align*}
Therefore, for every $M$ there exist $N_{12},N_{13},N_{23}>>0$ such that
\begin{align*}
&(z_1-z_2)^{N_{1 2}} (z_1-z_3)^{N_{1 3}} (z_2-z_3)^{N_{2 3}} 
Y(z_1) (1 \otimes Y(z_2)) (1 \otimes 1 \otimes Y(z_3)) \cdot \\
& \left(\mc S^{1 2}(z_1 - z_2) \mc S^{1 3}(z_1 - z_3) \mc S^{2 3}(z_2 - z_3) \right.\\
&\left. - \mc S^{2 3}(z_2 - z_3) \mc S^{1 3}(z_1 - z_3) \mc S^{1 2}(z_1 - z_2) \right) 
(a \otimes b \otimes c\otimes\vac)
\equiv 0 \mod h^M
\end{align*}
in $V_h((z_1))((z_2))((z_3))$.
On the other hand, multiplication by $z_1-z_2$, $z_1-z_3$ or $z_2-z_3$
is injective in $V_h((z_1))((z_2))((z_3))$, and commutes with multiplication by $h$.
Hence, 
\begin{align*}
& Y(z_1) (1 \otimes Y(z_2)) (1 \otimes 1 \otimes Y(z_3))
\left(\mc S^{1 2}(z_1 - z_2) \mc S^{1 3}(z_1 - z_3) \mc S^{2 3}(z_2 - z_3) \right.\\
&\left. - \mc S^{2 3}(z_2 - z_3) \mc S^{1 3}(z_1 - z_3) \mc S^{1 2}(z_1 - z_2) \right) 
(a \otimes b \otimes c\otimes\vac)
\end{align*}
vanishes modulo $h^M$ for every $M$,
i.e. vanishes identically.
\end{proof}
\begin{remark}\label{rem:ek-braided}
Etingof and Kazhdan define in \cite{EK5} the braided vertex algebra 
as free $\mb K[[\partial]]$-module
with a topological state-field correspondence $Y$
and a braiding $\mc S$
satisfying the $\mc S$-locality \eqref{S-locality}
and (b), (e) and (f) of Proposition \ref{prop:non-deg}.
\end{remark}
\begin{remark}\label{EK5Cor.1.3}
As pointed out in \cite[Cor.1.3]{EK5},
in a braided vertex algebra the assumption that
$Y(z)(T\otimes1)=\spd z Y(z)$ is not required
if the left shift condition holds.
Indeed, in the proof of Lemma \ref{EK5Lem.1.2} this assumption is not needed.
Moreover, by the left shift condition and Taylor expansion, 
one proves that $\iota_{z, u} \mc S(z+u) = \left( e^{-uT} \otimes 1 \right) \mc S(z) \left( e^{uT} \otimes 1\right)$. 
Using Lemmas \ref{EK5Lem.1.2} and \ref{QK15Lem.1.2b}, the following equalities follow:
\begin{align*}
Y(z) \mc S(z) \left( e^{uT} \otimes 1 \right) 
= e^{zT} Y(-z) \left( 1 \otimes e^{uT} \right) (1\,2) 
= e^{(z+u)T} \iota_{z,u}Y(-z-u) (1 \,2)\\
= \iota_{z,u} \left( Y(z+u) \mc S(z+u) \right) 
= \iota_{z,u} Y(z+u) \left( e^{-uT} \otimes 1 \right) \mc S(z) \left( e^{uT} \otimes 1\right).
\end{align*}
Since $\mc S(z)  \left( e^{uT} \otimes 1 \right)$ is invertible, one has 
$$
Y(z) = \iota_{z,u} Y(z+u) \left( e^{-uT} \otimes 1 \right)
$$ 
from which 
$$
Y(z)\left( e^{uT} \otimes 1 \right) = \iota_{z,u} Y(z+u) = e^{u \spd{z}} Y(z)
\,.
$$ 
The claim follows by taking the coefficient of $u$ in both sides of the above equation.
\end{remark}

%%%
\subsection{Quantum vertex algebras}
\label{sec:3.3}

\begin{definition}
\label{QVADef}
A \emph{quantum vertex algebra} is a braided vertex algebra satisfying the \emph{associativity relation} (cf. \eqref{fAssociativity}): for any $a, b, c \in V$ and $M \in \mb Z_{\geq 0}$ there exists $N \in \mb Z_{\geq 0}$
such that
\begin{equation}
\label{hfAssociativity}
\begin{split}
\gse{z}{w} & (z+w)^N Y(z+w) \big(1\otimes Y(w) \big) (a \otimes b \otimes c)\\
&= (z+w)^N Y(w) \big( Y(z)\otimes 1 \big) (a \otimes b \otimes c)\ \textrm{mod}\ h^M.
\end{split}
\end{equation}
\end{definition}

\begin{proposition}[{\cite[Prop.1.4]{EK5}}]
\label{EK5Prop.1.4}
If a braided vertex algebra $(V, \vac, Y, \mc S)$ satisfies the following \emph{hexagon relation}:
\begin{equation}
\label{HexagonRelation}
\mc{S}(w)\big( Y(z)\otimes 1 \big) = (Y(z)\otimes 1)\Sp{2}{3}(w)\gse{w}{z}\Sp{1}{3}(z+w)
\,,
\end{equation}
then the  associativity relation \eqref{hfAssociativity} holds.
Consequently, $V$ is a quantum vertex algebra.
\end{proposition}
\begin{proof}
The associativity relation \eqref{hfAssociativity} is an immediate consequence 
of the quasi-associativity \eqref{hfQuasiAssociativity}
and the hexagon relation \eqref{HexagonRelation}.
\end{proof}

\begin{proposition}
\label{July2017ThmRem1}
Let $(V, \vac, T, Y, \mc S)$ be a quantum vertex algebra.
Then the hexagon relation \eqref{HexagonRelation}
holds modulo $\ker Y(z)$.
\end{proposition}
\begin{proof}
By Proposition \ref{EK5Prop.1.1},
we have the quasi-associativity relation \eqref{hfQuasiAssociativity}.
Comparing it with the associativity relation \eqref{hfAssociativity}, 
we get
\begin{align*}
&(z+w)^NY(w)\mc{S}(w)(Y(z) \otimes 1) (a \otimes b \otimes c)\\
&= (z+w)^NY(w)(Y(z)\otimes 1) \Sp{2}{3}(w)\Sp{1}{3}(z+w) (a \otimes b \otimes c)
\quad mod\ h^M
\,,
\end{align*}
for $N$ large enough.
Both the LHS and the RHS above, module $h^M$,
lie in the space $V((w))((z))$.
Hence, we can multiply by $\gse{w}{z}(z+w)^{-N}$, to get
\begin{align*}
&Y(w)\mc{S}(w)(Y(z) \otimes 1) (a \otimes b \otimes c)\\
&= \gse{w}{z} Y(w)(Y(z)\otimes 1) \Sp{2}{3}(w)\Sp{1}{3}(z+w) (a \otimes b \otimes c)
\quad mod\ h^M
\,.
\end{align*}
Since the above equation holds modulo $h^M$ for every $M$, it holds identically.
\end{proof}

\begin{remark}
Etingof and Kazhdan define in \cite{EK5} a quantum vertex algebra
as a braided vertex algebra (in their sense, cf. Remark \ref{rem:ek-braided})
satisfying the hexagon relation \eqref{HexagonRelation}.
They also construct an example of a vertex algebra for which
the kernel of the map $Y(z)$ is zero
(as well as the map in Proposition \ref{prop:non-deg}(f)).
\end{remark}

%%%
\subsection{Locality properties of braided and quantum vertex algebras}
\label{sec:3.35}

Let $V$ be a topologically free $\Kh$-module. 
An $\End_{\Kh}V$-valued formal distribution $a(z,w)$ is called \emph{local} if, 
for any $M \in \mb Z_{\geq0}$, there exists $N \in\mb Z_{\geq 0}$ such that
\begin{equation}
(z-w)^N a(z,w) = 0\ \textrm{mod}\ h^M.
\end{equation}
Similarly, $a(z,w)$ is called \emph{local on the element} $b \in V$ if, 
for any $M \in \mb Z_{\geq0}$, there exists $N \in\mb Z_{\geq 0}$ 
(possibly depending on $b$), such that
\begin{equation}
(z-w)^N a(z,w)b = 0\ \textrm{mod}\ h^M.
\end{equation}
Two $\End_{\Kh}V$-valued quantum fields $a(z)$ and $b(w)$ are called local (respectively local on the vector $c \in V$) if $[a(z),b(w)]$ (respectively $[a(z),b(w)]c$) is local.

Proposition \ref{BKProp.4.1} still holds for
topologically free $\mb{K}[[h]]$-modules:
\begin{proposition}
\label{BKProp.4.1h}
Let $V$ be a topologically free $\mb K[[h]]$-module,
let $\vac\in V$,
and let $Y:\,V\otimes V\to V_h((z))$ 
satisfy the vacuum axiom (v) and translation covariance axiom (vi) of Definition \ref{BVADef}.
Then $Y$ satisfies the associativity relation \eqref{hfAssociativity} if and only if 
all pairs $\big( Y(z)(a\otimes-), Y^{op}(w)(b\otimes-) \big)$ are local on each element of $V$.
\end{proposition}
\begin{proof}
The locality assumption on $\big( Y(z)(a\otimes-), Y^{op}(w)(b\otimes-) \big)$
says that
$$
(z-w)^NY(z)(1\otimes Y^{op}(w))(a\otimes b\otimes c)
=
(z-w)^NY^{op}(w)(1\otimes Y(z))(b\otimes a\otimes c)
\,\text{mod}\,h^M
\,,
$$
for some $N \in \mb Z_{\geq 0}$ depending on $a,b,c\in V,\,M\in\mb Z_{\geq0}$.
By Definition \eqref{eq:Yop} of $Y^{op}$, the above equation can be rewritten as
\begin{align*}
& (z-w)^NY(z)(1\otimes e^{wT}Y(-w))(a\otimes c\otimes b) \\
& =
(z-w)^Ne^{wT}Y(-w)(Y(z)\otimes1)(a\otimes c\otimes b)
\,\text{mod}\,h^M
\,.
\end{align*}
By Lemma \ref{QK15Lem.1.2b} this is equivalent to
\begin{align*}
& (z-w)^Ne^{wT}\iota_{z,w}Y(z-w)(1\otimes Y(-w))(a\otimes c\otimes b) \\
& =
(z-w)^Ne^{wT}Y(-w)(Y(z)\otimes1)(a\otimes c\otimes b)
\,\text{mod}\,h^M
\,,
\end{align*}
which is the same as the associativity relation \eqref{hfAssociativity},
after changing the sign of $w$ and exchanging $b$ and $c$.
\end{proof}

The following proposition may be viewed as a quantum analogue of the Goddard's uniqueness Theorem (cf. Lemma \ref{Goddard}).
\begin{proposition}
\label{BraidedGoddard}
Let $V$ be a braided vertex algebra 
and let $a(z)$ be a translation covariant $\End_{\mb K[[h]]}V$-valued quantum field, 
such that, for every $b\in V$,
the pair of quantum fields $(a(z),Y(z)(b \otimes -))$ 
is local on every vector of $V$.
Then
\begin{equation}\label{eq:matteo}
a(z) = Y(z)\mc{S}(z)(a \otimes -) 
\,\,\big(= Y^{op}(z)(a \otimes -)\,\big)
\end{equation}
where $a=\Res_z z^{-1}a(z)\vac$.
\end{proposition}
\begin{proof}
By assumption, for every $b \in V$ and $M \in \mb Z_{\geq0}$, 
there exists $N_1 \in \mb Z_{\geq0}$ such that
\begin{displaymath}
(z-w)^{N_1} a(z)Y(w) (b \otimes \vac) = (z-w)^{N_1} Y(w) \big( b \otimes a(z)\vac \big)\ \textrm{mod}\ h^M.
\end{displaymath}
By Lemma \ref{QK15Lem.1.2}, $a(z)\vac = e^{zT}a = Y(z)(a \otimes \vac)$ from which
\begin{equation}
\label{AGT1}
(z-w)^{N_1} a(z)Y(w)(b\otimes\vac) 
= (z-w)^{N_1} Y(w)\big( 1 \otimes Y(z) \big) (b \otimes a \otimes \vac)\ \textrm{mod}\ h^M.
\end{equation}
On the other hand, by $\mc{S}$-locality \eqref{S-locality}, 
there exists $N_2 \in \mb Z_{\geq0}$ such that
\begin{equation}
\begin{split}
\label{AGT2}
(z-w)^{N_2} &Y(z) \big( 1 \otimes Y(w) \big) \big( \mc{S}(z-w)(a \otimes b)\otimes \vac \big)\\
&=(z-w)^{N_2} Y(w) \big(1 \otimes Y(z)\big)(b \otimes a \otimes \vac)\ \textrm{mod}\ h^M.
\end{split}
\end{equation}
Combining equations \eqref{AGT1} and \eqref{AGT2}, 
we have, for $N = \max\{N_1,N_2\}$,
\begin{equation}
\begin{split}
\label{AGT3}
(z-w)^{N} &Y(z)(1 \otimes Y(w))(\mc{S}(z-w)(a \otimes b)\otimes \vac)\\
&=(z-w)^{N} a(z)Y(w) (b \otimes \vac)\ \textrm{mod}\ h^M.
\end{split}
\end{equation}
Note that $Y(w)(b\otimes\vac)$ has no negative powers of $w$, by Lemma \ref{QK15Lem.1.2}.
Hence, the RHS of \eqref{AGT3} has no negative powers of $w$
and can be evaluated at $w=0$.
Likewise, in the LHS $Y(w)(\cdot\,\otimes\vac)$ has no negative powers of $w$,
and, for $N$ large enough, there are no negative powers of $(z-w)^N$ as well.
We thus get, letting $w=0$ and using Lemma \ref{QK15Lem.1.2},
\begin{align*}
z^{N} Y(z)\mc{S}(z)(a \otimes c) = z^{N} a(z)c\ \textrm{mod}\ h^M\,,
\end{align*}
which obviously implies
\begin{equation}
\label{AGT4}
Y(z)\mc{S}(z)(a \otimes b) = a(z)b\ \textrm{mod}\ h^M.
\end{equation}
Since equation \eqref{AGT4} holds for every $M \in \mb Z_{\geq0}$,
we get the claim \eqref{eq:matteo}.
\end{proof}

\begin{corollary}\label{qgoddard}
Let $V$ be a quantum vertex algebra.
Let $a(z)$ be a translation covariant $\End_{\Kh}V$-valued quantum field.
Then each pair of quantum fields 
$$
(a(z),Y(z)(b \otimes -))
\,\,,\,\,\,\,
b\in V\,,
$$
is local on every element of $V$
if and only if $a(z)=Y^{op}(z)(a\otimes-)$,
where $a=a(z)\vac|_{z=0}$.
\end{corollary}
\begin{proof}
The ``only if'' part is given by Proposition \ref{BraidedGoddard},
while the ``if'' part is a consequence of Propositions \ref{EK5Prop.1.4} 
and \ref{BKProp.4.1h}.
\end{proof}

%%%
\section{$\mc{S}$-commutative quantum vertex algebras}
\label{S-commutativeBVAs}

\begin{definition}
A braided, or quantum, vertex algebra $V$
is called $\mc{S}$-commutative if the $\mc S$-locality \eqref{S-locality} holds for $N=0$:
\begin{equation}
\label{S-commutativity}
\begin{split}
Y(z) &\big( 1\otimes Y(w) \big)\big(\gse{z}{w}\mc{S}(z-w)(a \otimes b) \otimes c \big)\\
&= Y(w) \big( 1\otimes Y(z) \big)(b \otimes a \otimes c).
\end{split}
\end{equation}
\end{definition}

\begin{remark}
Note that, if $V=W[[h]]$, then 
$Y(z) \big( 1\otimes Y(w) \big)\big(\gse{z}{w}\mc{S}(z-w)\otimes 1)$
has values in $(W((z))((w))\otimes\mb K((z-w)))[[h]]$.
Expanding via $\iota_{z,w}:\,\mb K((z-w))\to\mb K((z))[[w]]$,
we get an element of $W((z))((w))[[h]]$,
while if we expand via $\iota_{w,z}$ we may have divergences.
\end{remark}

\begin{theorem}
\label{ScommutativeBVAHolomBVA}
Let $V$ be a braided vertex algebra.
Then $V$ is $\mc{S}$-commutative if and only if 
$Y(z)(a \otimes b) \in V[[z]]$ for any $a, b \in V$.
\end{theorem}
\begin{proof}
Assume that $V$ is $\mc S$-commutative.
Taking $c=\vac$ in equation \eqref{S-commutativity}, 
multiplying both sides by $z^{-1}$ and taking the residue $Res_z$, we get
\begin{equation}\label{0124:eq1}
Y(w)(b \otimes a) = Res_z \Big(z^{-1} Y(z) \big(1 \otimes e^{wT} \big) \gse{z}{w} \mc{S}(z-w) (a \otimes b) \Big)
\in V[[w]]
\,,
\end{equation}
proving the ``only if'' part.
For the ``if'' part, we have, by assumption,
\begin{equation}\label{eq:notte1}
Y(w) \big( 1 \otimes Y(z) \big)(b \otimes a \otimes c) \in V[[z, w]]\subset V((z))((w))
\end{equation}
and 
\begin{equation}\label{eq:notte2}
Y(z) \big( 1\otimes Y(w) \big)\big(\iota_{z,w}\mc{S}(z-w)(a \otimes b) \otimes c \big) 
\in V((z))((w))
\,\text{ mod }\, h^{M}
\,,
\end{equation}
for any $M \in \mb Z_{\geq0}$. 
The $\mc S$-locality equates \eqref{eq:notte1} and \eqref{eq:notte2}
multiplied by $(z-w)^N$, for some $N\in\mb Z_{\geq0}$.
On the other hand, 
multiplication by $(z-w)^N$ has zero kernel in $V((z))((w))$.
As a consequence, \eqref{S-commutativity} holds mod $h^M$ for every $M$.
Hence \eqref{S-commutativity} holds.
\end{proof}

\begin{definition}
The \emph{normally ordered product}
$V \otimes V \rightarrow V$, $a\otimes b\mapsto\,{:}ab{:}$, 
on a braided vertex algebra $V$ is given by:
\begin{equation}\label{eq:nop}
{:}ab{:}\,=Res_z \big( z^{-1} Y(z)(a \otimes b) \big).
\end{equation}
\end{definition}
Note that, writing $Y(z)(a \otimes b) = \sum_{n \in \mb Z} a_{(n)}b\, z^{-n-1}$, we see that ${:}ab{:} = a_{(-1)}b$.

\begin{lemma}\label{lem:cauchy}
In an $\mc S$-commutative braided vertex algebra $V$, we have
\begin{equation}\label{eq:cauchy}
Y(z)(a\otimes b)=\,{:}(e^{zT}a)b{:}
\,.
\end{equation}
\end{lemma}
\begin{proof}
By Theorem \ref{ScommutativeBVAHolomBVA}, $Y(z)(a\otimes b)$ 
is a formal power series in $z$ satisfying the Cauchy problem
$$
\spd{z} Y(z)(a\otimes b)=Y(z)((Ta)\otimes b)
\,\,,\,\,\,\,
Y(z)(a\otimes b)|_{z=0}=\,{:}ab{:}
\,.
$$
The RHS of \eqref{eq:cauchy} obviously solves this problem.
The claim follows.
\end{proof}

\begin{theorem}\label{thm:scomm-bva}
Let $V$ be an $\mc{S}$-commutative quantum vertex algebra. 
Then the normally ordered product \eqref{eq:nop} is a unital (with unity $\vac$), associative,
differential (with derivation $T$) product, 
and the following commutation relation holds ($a,b\in V$):
\begin{equation}
\label{ScommutativeBVA+AssociativityCommRelations}
{:}ba{:}\, = \,{:}Res_z z^{-1} (e^{zT} \otimes 1) \mc{S}(z) (a \otimes b) {:}
\,.
\end{equation}
\end{theorem}
\begin{proof}
By axiom (i) of Definition \ref{def:state-field},
$\vac$ is the unit element of the normally ordered product \eqref{eq:nop},
while, by axiom (ii), $T$ is a derivation
of this product.
Next, by Theorem \ref{ScommutativeBVAHolomBVA} 
all $Y(z)(a\otimes b)$ have only non-negative powers of $z$.
Hence, we can set $N=0$ in the associativity relation \eqref{hfAssociativity}.
Multiplying its both sides by $z^{-1}w^{-1}$ and taking residues in $z$ and $w$,
we get
$( a_{(-1)} b)_{(-1)} c = a_{(-1)} ( b_{(-1)} c )\ \textrm{mod}\ h^M$ 
for every $M \in \mb Z_{\geq0}$. 
As a consequence, $(a_{(-1)} b)_{(-1)} c = a_{(-1)} ( b_{(-1)} c )$,
i.e. the normally ordered product is associative.
%
%We are left to prove the commutation relation 
%\eqref{ScommutativeBVA+AssociativityCommRelations}.
%
Finally, multiplying both sides of \eqref{0124:eq1} 
by $w^{-1}$ and taking the residue $Res_w$, we get
\begin{equation}\label{eq:4.8}
{:}ba{:}\, = Res_z \big( z^{-1} Y(z) \mc{S}(z) (a \otimes b) \big)
\,.
\end{equation}
Equation \eqref{ScommutativeBVA+AssociativityCommRelations}
follows from \eqref{eq:4.8} and Lemma \ref{lem:cauchy}.
\end{proof}
\begin{remark}
Let $V$ be a unital associative algebra with a derivation $T$ over $\mb K[[h]]$,
and a fixed braiding map 
$\mc S(z):\,V\widehat{\otimes} V\to V\widehat{\otimes} V\widehat{\otimes}\big(\mb K((z))[[h]]\big)$.
Then formula \eqref{eq:cauchy} defines on $V$ a structure of a field algebra (see Definition \ref{fieldAlgebraDef}) over $\mb K[[h]]$. 
The commutation relation \eqref{ScommutativeBVA+AssociativityCommRelations} is necessary for having a commutative quantum vertex algebra. 
We can only prove that, 
if the left and right shift conditions on $\mc S$ hold (conditions (b) and (c)
from Proposition \ref{prop:non-deg}),
then \eqref{ScommutativeBVA+AssociativityCommRelations} is sufficient,
provided that $\mc S(z)$ has no negative powers in $z$.
\end{remark}
\begin{remark}
Using Lemma \ref{lem:cauchy},
it is not hard to check that
in an $\mc S$-commutative quantum vertex algebra,
the hexagon relation \eqref{HexagonRelation} becomes
$$
S(w)({:}ab{:}\otimes c)
=
({:}\,\,\,\,{:}\otimes 1)\mc S^{23}(w)\mc S^{13}(w) (a \otimes b \otimes c),\,\, a, b , c \in V
\,.
$$
\end{remark}

%%%
\section{Characterizations of quantum vertex algebras}
\label{sec:5}

%%%
\subsection{Sufficient conditions for $\mc S$-locality}
\label{sec:5.1}

The following theorem is the quantum analogue of Theorem \ref{BKThm.7.3}. 
It first appeared in \cite[Prop.2.19]{L10} with a different proof.
\begin{theorem}\label{July2017Thm}
Let $(V,\vac,Y,\mc S)$ be a braided state-field correspondence.
Suppose moreover that the associativity relation \eqref{hfAssociativity} holds.
Then the following two conditions are equivalent:
\begin{enumerate}[(a)]
\item
the $\mc S$-locality \eqref{S-locality} holds;
\item
$Y\mc{S} = Y^{op}$, where $Y^{op}$ is defined by \eqref{eq:Yop}.
\end{enumerate}
\end{theorem}
\begin{proof}
If (a) holds, by the $\mc{S}$-locality \eqref{S-locality} and Lemma \ref{EK5Lem.1.2}
we get $Y\mc{S} = Y^{op}$, proving (b).
Let us prove that (b) implies (a).
By Proposition \ref{BKProp.4.1h}, 
the associativity relation \eqref{hfAssociativity} 
implies the locality of $Y$ and $Y^{op}$ on every element:
for every $a, b, c \in V$, there exists $L \geq 0$, such that
\begin{equation}
\label{THMf1}
\begin{split}
(u-w)^L &Y(u) \big( 1 \otimes Y^{op}(w) \big)(b \otimes c \otimes a)\\
&= (u-w)^L Y^{op}(w) \big( 1 \otimes Y(u) \big)(c \otimes b \otimes a)\ mod\ h^M.
\end{split}
\end{equation}
Moreover, according to Lemma \ref{BKLem.3.8}, the equation still holds if we replace 
$a$ by $Ta$, hence by $e^{zT}a$. 
Therefore, one has
\begin{equation}
\label{THMf1e}
\begin{split}
(u-w)^L &Y(u) \big( 1 \otimes Y^{op}(w) \big)(b \otimes c \otimes e^{zT}a)\\
&= (u-w)^L Y^{op}(w) \big( 1 \otimes Y(u) \big)(c \otimes b \otimes e^{zT}a)\ mod\ h^M.
\end{split}
\end{equation}
By the definition of $Y^{op}$ and Proposition \ref{BKProp.2.7} we get
\begin{equation}
\begin{split}
\label{THMf2}
&(u-w)^L\ e^{wT}\gse{u}{w}Y(u-w)\ \big( 1 \otimes \gse{w}{z}Y(z-w) \big) (b \otimes a \otimes c)\\
&= (u-w)^L\ e^{wT}\gse{w}{z} Y(z-w) \big( \gse{u}{z}Y(u-z) \otimes 1 \big)(b \otimes a \otimes c)\ mod\ h^M.
\end{split}
\end{equation}
Since $Y(z-w)(a \otimes c)$ is a Laurent series ($mod\ h^M$) in $z-w$, there exists $P \in \mb Z_{\geq0}$ such that $(z-w)^P\ Y(z-w)(a \otimes c) \in V[[z-w]]\ mod\ h^M$. Multiplying both sides of equation \eqref{THMf2} for $(z-w)^P$ we thus obtain an expression which is regular in $w$. Putting $w = 0$ we obtain
\begin{equation}
\label{THMf3}
\begin{split}
Y(u)(1& \otimes Y(z))(b\otimes a\otimes c) = z^{-P} u^{-L} \Big((z-w)^P (u-w)^L e^{wT} \cdot\\
&\cdot \gse{w}{z} Y(z-w)\big(\gse{u}{z}Y(u-z) \otimes 1 \big)(b \otimes a \otimes c) \Big)\big|_{w = 0}\ mod\ h^M.
\end{split}
\end{equation}
By axioms (v) and (vi) of Definition \ref{BVADef},
Proposition \ref{BKProp.2.7} and the definition of $Y^{op}$, we get
\begin{equation}
\label{THMf4}
\begin{split}
&\gse{u}{z}Y(u-z)(b \otimes a) = Y(u) \big(e^{-zT}b \otimes a \big) = e^{uT}Y^{op}(-u) \big(a \otimes e^{-zT}b \big)\\
&= e^{(u-z)T} \gse{u}{z}Y^{op}(z-u)(a \otimes b) = e^{(u-z)T} \gse{u}{z}Y(z-u) \mc{S}(z-u)(a \otimes b).
\end{split}
\end{equation}
In the last equality we used the assumption (b).
Combining equations \eqref{THMf3} and \eqref{THMf4} 
and applying again Proposition \ref{BKProp.2.7}(c), we obtain
\begin{equation}
\label{RHSS-loc}
\begin{split}
Y(u)(1& \otimes Y(z))(b\otimes a\otimes c)
= 
z^{-P} u^{-L} \Big((z-w)^P (u-w)^L e^{wT} 
\gse{w}{z} Y(z-w) \cdot \\
&\quad \cdot \big(e^{(u-z)T}\otimes1\big)
\big(
\gse{u}{z}Y(z-u) \mc{S}(z-u)\otimes 1
\big)(a \otimes b \otimes c) 
\Big)\big|_{w = 0} \\
&= 
z^{-P} u^{-L} \Big((z-w)^P (u-w)^L e^{wT} 
\gse{w}{z} \gse{z-w}{u-z} \gse{u}{z}
Y(u-w) \cdot \\
&\quad \cdot \big(
Y(z-u) \mc{S}(z-u)\otimes 1
\big)(a \otimes b \otimes c) 
\Big)\big|_{w = 0}\ mod\ h^M
\,.
\end{split}
\end{equation}
It is easy to check that,
$$
\gse{w}{z} \gse{z-w}{u-z} \gse{u}{z}f(u-w,z-u)
=
\gse{u}{z} \gse{w}{u} f(u-w,z-u)
\, \text{ for every } f.
$$
Hence, \eqref{RHSS-loc} gives
\begin{equation}
\label{RHSS-loc2}
\begin{split}
Y(u)(1& \otimes Y(z))(b\otimes a\otimes c)
= 
\gse{u}{z}
z^{-P} u^{-L} \Big((z-w)^P (u-w)^L e^{wT} 
\gse{w}{u}
Y(u-w) \cdot \\
&\quad \cdot \big(
Y(z-u) \mc{S}(z-u)\otimes 1
\big)(a \otimes b \otimes c) 
\Big)\big|_{w = 0}\ mod\ h^M
\,.
\end{split}
\end{equation}
This expression will give the RHS of the $\mc S$-locality, which we want to prove.

For the LHS, let us rewrite formula \eqref{THMf1} 
with $z$ instead of $u$ and $a$, $b$ swapped:
\begin{equation}
\label{THMf1EK5}
\begin{split}
&(z-w)^L\ Y(z) (1 \otimes Y^{op}(w))(a \otimes c \otimes b)\\
&=  (z-w)^L\ Y^{op}(w) (1 \otimes Y(z)) (1\ 2)(a \otimes c \otimes b)\ mod\ h^M.
\end{split}
\end{equation}
By Lemma \ref{BKLem.3.8},
equation \eqref{THMf1EK5} still holds, with the same value of $L$,
if we replace $b$ by $e^{uT}b$.
Moreover, 
equation \eqref{THMf1EK5} also holds,
after increasing appropriately the value of $L$,
if we replace $a\otimes c\otimes b$ by $(z-u)^Q\mc S^{13}(z-u)(a\otimes c\otimes b)$,
for some large enough positive integer $Q$.
Indeed, $(z-u)^Q\mc S^{13}(z-u)(a\otimes c\otimes b)$,
modulo $h^M$,
is a finite sum of elements of the form
$a_i\otimes c\otimes b_i\otimes f_i(z-u)$,
where $f_i(z-u)\in\mb K[[z-u]]$.
Hence, we get
\begin{equation}
\label{THMf1EK5S}
\begin{split}
&(z-w)^L(z-u)^Q\ Y(z) (1\otimes Y^{op}(w))
(1\!\otimes\!1\!\otimes e^{uT})
\mc S^{13}(z-u)(a \otimes c \otimes b)\\
&=  (z-w)^L(z-u)^Q\ Y^{op}(w) (1 \otimes Y(z))
(1\!\otimes\!1\!\otimes e^{uT}) 
\mc S^{23}(z-u)(c\otimes a \otimes b)\ mod\ h^M.
\end{split}
\end{equation}
Applying the definition of $Y^{op}$, equation \eqref{THMf1EK5S} becomes
\begin{equation}
\label{THMf1EK5T}
\begin{split}
&(z-w)^L(z-u)^Q\ Y(z) (1\otimes e^{wT}Y(-w))
(1\!\otimes e^{uT}\otimes\!1)
\mc S^{12}(z-u)(a \otimes b \otimes c)\\
&=\!  (z\!-\!w)^L(z\!-\!u)^Q\!\ e^{wT}Y(-w) (Y(z) \!\otimes\! 1)
(1\!\otimes\! e^{uT}\!\otimes\!1) 
\mc S^{12}(z\!-\!u)(a \!\otimes\! b\!\otimes\! c)\ mod\, h^M.
\end{split}
\end{equation}
Next, we apply (twice) Proposition \ref{BKProp.2.7} (b) and (c) on both sides,
to get
\begin{equation}
\label{THMf2S}
\begin{split}
&(z-w)^{L}(z-u)^Q e^{wT} \gse{z}{w}Y(z-w)(1 \otimes \gse{w}{u}Y(u-w))\Sp{1}{2}(z-u)(a \otimes b \otimes c)\\
&=  (z\!-\!w)^{L}(z\!-\!u)^Q e^{wT} \gse{w}{u}Y(u\!-\!w) 
(\gse{z}{u}Y(z\!-\!u)\mc{S}(z\!-\!u)\!\otimes\! 1) (a \!\otimes\! b \!\otimes\! c)\ mod\, h^M.
\end{split}
\end{equation}
After multiplying by a sufficiently large integer power $(u-w)^R$,
the LHS of \eqref{THMf2S} has no negative powers in $w$, modulo $h^M$.
Hence, we can let $w=0$, to get
\begin{equation}
\label{LHSS-loc}
\begin{split}
&(z-u)^Q\ Y(z)(1 \otimes Y(u))\Sp{1}{2}(z-u)(a \otimes b \otimes c)
=  (z-u)^Q\gse{z}{u} u^{-R}\ z^{-L} \cdot\\
&\cdot \big((u\!-\!w)^R(z\!-\!w)^{L} e^{wT}
\gse{w}{u}Y(u\!-\!w)(Y(z\!-\!u)\mc{S}(z\!-\!u)\!\otimes\! 1) (a \!\otimes\! b \!\otimes\! c)\big)
\big|_{w = 0}\ mod\, h^M.
\end{split}
\end{equation}

Note that equation \eqref{RHSS-loc2} holds for every $P$ and $L$ sufficiently large.
Likewise, equation \eqref{LHSS-loc} holds for every $L$ and $R$ sufficiently large.
We can thus take all these exponents equal to a sufficiently large integer $D$.
As a result, we get, for \eqref{RHSS-loc2}
\begin{equation}
\label{RHSS-loc3}
\begin{split}
Y(u)(1& \otimes Y(z))(b\otimes a\otimes c)
= 
\gse{u}{z}
z^{-D} u^{-D} \Big((z-w)^D (u-w)^D e^{wT} 
\gse{w}{u}
Y(u-w) \cdot \\
&\quad \cdot \big(
Y(z-u) \mc{S}(z-u)\otimes 1
\big)(a \otimes b \otimes c) 
\Big)\big|_{w = 0}\ mod\ h^M
\,,
\end{split}
\end{equation}
and for \eqref{LHSS-loc}
\begin{equation}
\label{LHSS-loc3}
\begin{split}
&(z-u)^Q\ Y(z)(1 \otimes Y(u))\Sp{1}{2}(z-u)(a \otimes b \otimes c)
=  (z-u)^Q\gse{z}{u} u^{-D}\ z^{-D} \cdot\\
&\cdot \big((u\!-\!w)^D(z\!-\!w)^{D} e^{wT}
\gse{w}{u}Y(u\!-\!w)(Y(z\!-\!u)\mc{S}(z\!-\!u)\!\otimes\! 1) (a \!\otimes\! b \!\otimes\! c)\big)
\big|_{w = 0}\ mod\, h^M.
\end{split}
\end{equation}
Note that the RHS's of \eqref{RHSS-loc3} and \eqref{LHSS-loc3}
coincide, modulo $h^M$, after multiplying by a sufficiently large integer power of $z-u$,
depending only on $a$, $b$ and $M$.
We thus get
$$
(z-u)^N Y(u)(1 \otimes Y(z))(b\otimes a\otimes c)
\!=\!
(z-u)^N Y(z)(1 \otimes Y(u))\Sp{1}{2}(z-u)(a \otimes b \otimes c)
\ mod\, h^M,
$$
proving the $\mc S$-locality.
\end{proof}
\begin{corollary}
If $(V,\vac,Y,\mc S)$ is a quantum vertex algebra,
then $(V,\vac,Y^{op},\mc S^{-1})$ is a quantum vertex algebra as well.
\end{corollary}
\begin{proof}
By Proposition \ref{prop:2.10} $(V,\vac,Y^{op})$
is a field algebra, and $Y^{op}S^{-1}=Y=(Y^{op})^{op}$.
The claim follows by Theorem \ref{July2017Thm}.
\end{proof}

%%%
\subsection{Quantum $n$-products}
\label{sec:5.2}

Recall Definition \ref{nproducts} of $n$ products 
of quantum fields in a vertex algebras.
We define here their quantum analogue.
\begin{definition}
Let $(V,\vac,Y,\mc S)$ be a braided state-field correspondence.
For $n\in\mb Z$,
the \emph{quantum} $n$-\emph{product} $Y(z)^{\mc S}_{(n)}Y(z)$
is defined as follows:
\begin{equation}
\label{qnproductsDef}
\begin{split}
(Y(z)_{(n)}^{\mc{S}} Y(z))(a\!\otimes\! b\!\otimes\!c)
& = Res_x \Big(
\gse{x}{z}(x-z)^n
Y(x) \big( 1\otimes Y(z) \big) \mc{S}^{12}(x-z)(a\!\otimes\! b\!\otimes\!c) \\
& - \gse{z}{x}(x-z)^n 
Y(z) \big( 1\otimes Y(x) \big)(b \otimes a \otimes c)
\Big)
\,.
\end{split}
\end{equation}
Writing $(Y(z)_{(n)}^{\mc{S}} Y(z))(a \otimes b \otimes c) = (Y(z)_{(n)}^{\mc{S}} Y(z))(a \otimes b)c$, one can think of the quantum $n$-product as a formal distribution with coefficients in $\End_{\mb K[[h]]} V$.
\end{definition}

\begin{lemma}\label{lem:5.4}
The quantum $n$-product \eqref{qnproductsDef}
satisfies the following properties:
\begin{enumerate}[(a)]
\item
field condition ($a,b,c\in V$)
\begin{equation}\label{eq:qnp-field}
(Y(z)^{\mc S}_{(n)}Y(z))(a\otimes b\otimes c)\,\in V_h((z))\,;
\end{equation}
\item
vacuum condition ($a,b\in V$)
\begin{equation}\label{eq:qnp-vac}
(Y(z)^{\mc S}_{(n)}Y(z))(a\otimes b\otimes\vac)\,\in V[[z]]\,,
\end{equation}
\item
translation covariance
\begin{equation}\label{eq:qnp-tc}
T\circ (Y(z)^{\mc S}_{(n)}Y(z))
-(Y(z)^{\mc S}_{(n)}Y(z))\circ(1\otimes1\otimes T)
=
\partial_z(Y(z)^{\mc S}_{(n)}Y(z))
\,.
\end{equation}
\end{enumerate}
%
%In other words, $(Y(z)^{\mc S}_{(n)}Y(z))(a\otimes b\otimes-):\,V\to V_h((z))$
%is a translation covariant quantum field.
\end{lemma}

\begin{proof}
For part (a),
we consider separately the two expressions in the RHS of \eqref{qnproductsDef}.
In the first term,
$\iota_{x,z}(x-z)^n\mc S^{12}(x-z)(a\otimes b\otimes c)$
only involves non-negative powers of $z$.
Applying to it $1\otimes Y(z)$, we get, modulo $h^M$,
only finitely many negative powers of $z$.
For the second term,
$\iota_{z,x}(x-z)^n$ has non-negative powers of $x$,
hence only the negative powers of $x$ in $(1\otimes Y(x))(b\otimes a\otimes c)$
can give non-zero contribution to the residues in $x$,
which are finitely many modulo $h^M$.
Applying $Y(z)$ to these (finitely many) terms, we get, modulo $h^M$, 
an element of $V((z))$.
Claim (a) follows.

Claim (b) can be easily checked directly, but in fact 
it follows by (a), (c) and Lemma \ref{QK15Lem.1.2}.
We are left to prove claim (c).
By \eqref{qnproductsDef} we have

\begin{align*}
T(Y&(z)_{(n)}^{\mc{S}} Y(z))(a\!\otimes\! b\!\otimes\!c) \\
& = Res_x \Big(
\gse{x}{z}(x-z)^n
T Y(x) \big( 1\otimes Y(z) \big) \mc{S}^{12}(x-z)(a\!\otimes\! b\!\otimes\!c) \\
&\qquad - \gse{z}{x}(x-z)^n 
T Y(z) \big( 1\otimes Y(x) \big)(b \otimes a \otimes c)
\Big) \\
& = Res_x \Big(
\gse{x}{z}(x-z)^n
(\partial_x Y(x)) \big( 1\otimes Y(z) \big) \mc{S}^{12}(x-z)(a\!\otimes\! b\!\otimes\!c) \\
&\qquad + 
\gse{x}{z}(x-z)^n
Y(x) \big( 1\otimes (\partial_z Y(z)) \big) \mc{S}^{12}(x-z)(a\!\otimes\! b\!\otimes\!c) \\
&\qquad + 
\gse{x}{z}(x-z)^n
Y(x) \big( 1\otimes Y(z) \big) \mc{S}^{12}(x-z) (1\otimes1\otimes T)
(a\!\otimes\! b\!\otimes\!c) \\
&\qquad - \gse{z}{x}(x-z)^n 
(\partial_zY(z)) \big( 1\otimes Y(x) \big)(b \otimes a \otimes c)
\\
&\qquad - \gse{z}{x}(x-z)^n 
Y(z) \big( 1\otimes (\partial_xY(x)) \big)(b \otimes a \otimes c)
\\
&\qquad - \gse{z}{x}(x-z)^n 
Y(z) \big( 1\otimes Y(x) \big)(1\otimes 1\otimes T)(b \otimes a \otimes c)
\Big) 
\,.
\end{align*}
For the last equality we used (twice)
the translation covariance of $Y$.
Claim (c) is then obtained by integrating by parts in $x$.
\end{proof}

\begin{lemma}
For every $a,b\in V$ we have
\begin{equation}\label{eq:qprod1}
(Y(z)_{(n)}^{\mc{S}} Y(z))(a\otimes b\otimes\vac)
=
e^{zT}(a_{(n)}^{\mc{S}}b)\,,
\end{equation}
where
\begin{equation}
\label{eq:qprod2}
a_{(n)}^{\mc{S}}b 
= Res_z \big(z^{n}Y(z)S(z)(a \otimes b) \big) 
%= Res_z \big( z^{n}Y^{op}(z)(a \otimes b) \big)
\,.
\end{equation}
Consequently, the state corresponding to the quantum field $(Y(z)_{(n)}^{\mc{S}} Y(z))(a\otimes b\otimes -)$ is $a_{(n)}^{\mc{S}}b$.
\end{lemma}
\begin{proof}
By Lemmas \ref{lem:5.4} and \ref{QK15Lem.1.2}
it suffices to prove equation \eqref{eq:qprod1} at $z=0$.
We have
\begin{align*}
&(Y(z)_{(n)}^{\mc{S}} Y(z)) (a\otimes b\otimes\vac)|_{z=0}\\
& = Res_x \Big(
\gse{x}{z}(x-z)^n
Y(x) \big( 1\otimes Y(z) \big) \mc{S}^{12}(x-z)(a\!\otimes\! b\!\otimes\!\vac) \\
& - \gse{z}{x}(x-z)^n 
Y(z) \big( 1\otimes Y(x) \big)(b \otimes a \otimes\vac) \Big)\Big|_{z=0}
\,.
\end{align*}
Note that 
$Y(x)(a\otimes\vac)$
has no negative powers of $x$,
hence the residue in $x$ in the second term of the RHS vanishes.
As a result, we get, using Lemma \ref{QK15Lem.1.2},
\begin{align*}
(Y(z&)_{(n)}^{\mc{S}} Y(z)) (a\otimes b\otimes\vac)|_{z=0} \\
& = Res_x \Big(
\gse{x}{z}(x-z)^n
Y(x) \big( 1\otimes Y(z) \big) \mc{S}^{12}(x-z)(a\!\otimes\! b\!\otimes\!\vac)
\Big)\Big|_{z=0} \\
& = Res_x
x^n
Y(x) \mc{S}(x)(a\otimes b)
\,.
\end{align*}
\end{proof}

\begin{proposition}\label{prop:sn-prod}
In a braided vertex algebra $V$ we have the following properties:
\begin{enumerate}[(a)]
\item
$\vac_{(n)}^{\mc{S}}a = \delta_{n,-1}a$ for all $n\in\mb Z$ and $a\in V$;
\item
$a_{(n)}^{\mc{S}}\vac = \delta_{n\leq-1}\frac{T^{-n-1}}{(-n-1)!}a$ for all $n\in\mb Z$ and $a\in V$;
\item
$a_{(-n-1)}^{\mc{S}}b=\Big(\frac{T^n}{n!}a\Big)_{(-1)}^{\mc{S}}b$
for all $n\geq0$ and $a,b\in V$.
\item
$T$ is a derivation of all products \eqref{eq:qprod2}.
\end{enumerate}
\end{proposition}
\begin{proof}
Claims (a) and (b)
are a direct consequence of Proposition \ref{prop:non-deg}(a).
Indeed, 
since $(\mc S(z) - id)(\vac \otimes a) \in \ker Y(z)$, we have
\begin{align*}
\vac_{(n)}^{\mc S} a = Res_z z^n Y(z) \mc S(z) (\vac \otimes a) = Res_z z^n Y(z) (\vac \otimes a) = Res_z z^n a = \delta_{n, -1} a
\,,
\end{align*}
proving (a).
Similarly, since $(\mc S(z) - id)(a \otimes \vac) \in \ker Y(z)$, we have
\begin{align*}
&a_{(n)}^{\mc S} \vac = Res_z z^n Y(z) \mc S(z) (a \otimes \vac) = Res_z z^n Y(z) (a \otimes \vac) = Res_z z^n e^{zT}a\\
&= \delta_{n \leq -1} \frac{T^{-n-1}}{(-n-1)!}a
\,,
\end{align*}
proving (b).
Next, we prove claim (c).
Applying $Res_z z^{-n-1}$ to both sides of equation \eqref{EK5Lem.1.2Eq},
we get
\begin{align*}
a_{(-n-1)}^{\mc S} b = \sum_{l \geq 0} \frac{T^l}{l!} (-1)^{l+n} b_{(l-n-1)} a.
\end{align*}
By translation covariance, for any $n \geq 0$, one has
$$a_{(-n-1)} b = \left( \frac{T^n}{n!} a\right)_{(-1)} b.$$
It follows that, for any $n \geq 0$,
\begin{align*}
a_{(-n-1)}^{\mc S} b = \sum_{l = 0}^n (-1)^{l+n} \frac{T^l}{l!} \left(\frac{T^{n-l}}{(n-l)!} b \right)_{(-1)} a + \sum_{m \geq 0} (-1)^{m+1} \frac{T^{m+n+1}}{(m+n+1)!} b_{(m)} a.
\end{align*}
Finally, we prove claim (d).
Let $a, b \in V$. By the translation covariance on $Y$, one has
\begin{align*}
&T Y(z)\mc S(z) (a \otimes b) - Y(z)\mc S(z) (T \otimes 1)(a \otimes b) - Y(z)\mc S(z)(1 \otimes T) (a \otimes b)\\
&= Y(z) (1 \otimes T) \mc S(z) (a \otimes b) + Y(z) (T \otimes 1) \mc S(z) (a \otimes b) - Y(z)\mc S(z)(T \otimes 1)(a \otimes b)\\
&\phantom{{}={}} - Y(z)\mc S(z)(1 \otimes T)(a \otimes b)\\
&= Y(z) \left([T \otimes 1, \mc S(z)] + [1 \otimes T, \mc S(z)]\right)(a \otimes b).
\end{align*}
Hence, by Proposition \ref{prop:non-deg}(d),
\begin{equation}
\label{etichettaancoranontroppolunga}
T Y(z)\mc S(z) = Y(z)\mc S(z) (T \otimes 1) + Y(z)\mc S(z)(1 \otimes T)\,.
\end{equation}
The claim follows by taking residues of both sides multiplied by $z^n$.
\end{proof}

%%%
\subsection{Quantum Borcherds identity and quantum $n$-product identities}
\label{sec:5.3}

Let $(V, \vac, Y, \mc S)$ be a braided state-field correspondence.
The \emph{quantum Borcherds identity} is ($a, b, c \in V$, $n \in \mathbb{Z}$):
\begin{equation}
\label{qBorcherdsIdentities}
\begin{split}
& \gse{z}{w}(z-w)^n
Y(z) \big( 1\otimes Y(w) \big) \mc{S}^{12}(z-w)
(a \otimes b \otimes c)  \\
& - \gse{w}{z}(z-w)^n
Y(w) \big( 1\otimes Y(z) \big) (b \otimes a \otimes c)
= 
\sum_{j \in \mb Z_{\geq0}} Y(w)\big(a_{(n+j)}^{\mc{S}}b \otimes c\big)\
\frac{\spd{w}^j \delta(z,w)}{j!}
\end{split}
\end{equation}
The \emph{quantum} $n$-\emph{product identities}, for $n \in \mb Z$, 
are ($a, b, c \in V$)
\begin{equation}
\label{qnproductIdentities}
\Big( Y(z)_{(n)}^{\mc{S}} Y(z)\Big)(a\otimes b\otimes c) 
= Y(z)\big(a_{(n)}^{\mc{S}}b \otimes c \big).
\end{equation}

\begin{lemma}
\label{EquivalenceqBorcherdsIdentitiesqnprodIdentitiesLem}
The quantum Borcherds identity \eqref{qBorcherdsIdentities} 
is equivalent to 
the quantum $n$-product identities \eqref{qnproductIdentities} for all $n \in \mb Z$
and the $\mc{S}$-locality \eqref{S-locality} combined.
\end{lemma}

\begin{proof}
Let us first prove that the $\mc S$-locality \eqref{S-locality}
and the quantum $n$-product identities \eqref{qnproductIdentities}
imply the quantum Borcherds identities \eqref{qBorcherdsIdentities}.
Let
\begin{equation}
\label{eq:a}
\begin{split}
a(z,w)
& =
\gse{z}{w}(z-w)^n
Y(z) \big( 1\otimes Y(w) \big) \mc{S}^{12}(z-w)(a \otimes b\otimes c)
\\
& -\gse{w}{z}(z-w)^n
Y(w) \big( 1 \otimes Y(z) \big)(b \otimes a \otimes c )
\end{split}
\end{equation}
By the $\mc{S}$-locality assumption \eqref{S-locality},
this is a local quantum field in the sense of Section \ref{sec:3.35},
i.e. for every $M\in\mb Z_+$ there exists $N\in\mb Z_+$ such that
$$
(z-w)^Na(z,w)=0\quad mod\ h^M
\,.
$$
Note that the Decomposition Theorem \ref{KDecThm} holds for local fields
in a topologically free $\mb K[[h]]$-module $V$ in the following form
$$
a(z,w)
=
\sum_{j=0}^N c^j(w) \frac{\spd{w}^j \delta(z,w)}{j!}
\quad mod\ h^M
\,,
$$
where
$$
c_j(w)
=
\Res_z(z-w)^ja(z,w)
\,,
$$
which vanishes modulo $h^M$ for $j\geq N$.
Taking $a(z,w)$ from \eqref{eq:a}, we thus get
\begin{align*}
&
\gse{z}{w}(z-w)^n
Y(z) \big( 1\otimes Y(w) \big) \mc{S}^{12}(z-w)(a \otimes b\otimes c)
\\
& -\gse{w}{z}(z-w)^n
Y(w) \big( 1 \otimes Y(z) \big)(b \otimes a \otimes c ) 
= 
\sum_{j=0}^N c^j(w) \frac{\spd{w}^j \delta(z,w)}{j!} \quad mod\ h^M,
\end{align*}
where
\begin{align*}
c^j(w) &= Res_z \Big(
\gse{z}{w}(z-w)^{n+j}
Y(z) \big( 1\otimes Y(w) \big) \mc{S}^{12}(z-w)(a \otimes b\otimes c) \\
& -\gse{w}{z}(z-w)^{n+j}
Y(w) \big( 1 \otimes Y(z) \big)(b \otimes a \otimes c )
\Big) \\
&= \Big(Y(w)_{(n+j)}^{\mc{S}} Y(w)\Big)(a\otimes b\otimes c) \\
&= Y(w)\big(a_{(n+j)}^{\mc{S}}b \otimes c \big)
\,,
\end{align*}
which vanishes modulo $h^M$ for $j\geq N$.
For the last identity we used the assumption \eqref{qnproductIdentities}.
Hence,
\begin{align*}
&
\gse{z}{w}(z-w)^n
Y(z) \big( 1\otimes Y(w) \big) \mc{S}^{12}(z-w)(a \otimes b\otimes c)
\\
& -\gse{w}{z}(z-w)^n
Y(w) \big( 1 \otimes Y(z) \big)(b \otimes a \otimes c ) \\
& = 
\sum_{j\in\mb Z_{\geq0}} Y(w)\big(a_{(n+j)}^{\mc{S}}b \otimes c \big)
\frac{\spd{w}^j \delta(z,w)}{j!} \quad mod\ h^M
\,.
\end{align*}
Since the above identity holds modulo $h^M$ for every $M$,
it must hold identically,
i.e. \eqref{qBorcherdsIdentities} holds.

Conversely,
the quantum $n$-product identities \eqref{qnproductIdentities} are obtained from
the quantum Borcherds identity \eqref{qBorcherdsIdentities} by taking the residue in $z$. 
Moreover, for any $M \in \mb Z_{\geq0}$, 
$$
a_{(n)}^{\mc{S}} b = Res_z \big( z^{n} Y(z)\mc S(z)(a \otimes b) \big)
$$
vanishes modulo $h^M$ for sufficiently large $n$.
The $\mc S$-locality \eqref{S-locality} thus follows by the quantum Borcherds identity.
\end{proof}

%%%
\subsection{Equivalence of the quantum Borcherds identity 
and the $\mc S$-locality combined with associativity}
\label{sec:5.4}

By the following theorem, the quantum Borcherds identity holds for any quantum vertex algebra.

\begin{theorem}
\label{associativeBVABorcherdsIdentities}
Let $(V, \vac, Y, \mc S)$ be a braided state-field correspondence.
Then the $\mc{S}$-locality \eqref{S-locality} and the associativity relation \eqref{hfAssociativity} 
hold if and only if the quantum Borcherds identity \eqref{qBorcherdsIdentities} holds. 
Consequently, the quantum Borcherds identity \eqref{qBorcherdsIdentities} holds 
in any quantum vertex algebra.
\end{theorem}

\begin{lemma}
\label{EquivalenthfAssociativityLem}
If $(V, \vac, Y, \mc S)$ is a braided state-field correspondence,
%Assume that $Y\mc{S} = Y^{op}$.
the associativity relation \eqref{hfAssociativity} holds if and only if, 
for any $a, b, c \in V$ and $M \in \mb Z_{\geq0}$, there exists $N' \geq 0$ such that
\begin{equation}
\label{EquivalenthfAssociativity}
\begin{split}
(z+w)^{N'} &\gse{z}{w} Y(z+w) \big( 1 \otimes Y(w) \big) \big( \mc{S}(z)(a \otimes b) \otimes c \big)\\
&= (z+w)^{N'} Y(w) \big( Y(z)\mc S(z) \otimes 1 \big) (a \otimes b \otimes c)\ \textrm{mod}\ h^M.
\end{split}
\end{equation}
\end{lemma}

\begin{proof}
By definition of $\mc{S}$, there exist $s \in \mb Z_{\geq0}$, $K \in \mathbb{Z}$, $a_i, b_i \in V$ and ${f_i}_k \in \mb{K}$ for $i=1, \ldots, s$ and $k \geq K$, such that $\mc{S}(z)(a \otimes b) = \sum_{i=1}^s \sum_{k \geq K} a_i \otimes b_i {f_i}_k z^k$ mod $h^M$. If the associativity relation \eqref{hfAssociativity} holds, for any $M \in \mb Z_{\geq0}$ and $i=1, \ldots, s$, there exists $N_i \geq 0$ such that, for any $c \in V$,
\begin{equation}
\label{EquivalenthfAssociativityLemEq1}
\begin{split}
(z+w&)^{N_i} \gse{z}{w} Y(z+w) \big(1\otimes Y(w) \big) (a_i \otimes b_i \otimes c)\\
&= (z+w)^{N_i} Y(w) \big( Y(z)\otimes 1 \big) (a_i \otimes b_i \otimes c)\ \textrm{mod}\ h^M.
\end{split}
\end{equation}
In particular, for $N' \geq \max\{N_i\}_{i=1, \ldots, s}$ one has
\begin{equation}
\label{EquivalenthfAssociativityLemEq2}
\begin{split}
(z+w&)^{N'} \gse{z}{w} Y(z+w) \big(1\otimes Y(w) \big) (a_i \otimes b_i \otimes c)\\
&= (z+w)^{N'} Y(w) \big( Y(z)\otimes 1 \big) (a_i \otimes b_i \otimes c)\ \textrm{mod}\ h^M.
\end{split}
\end{equation}
The LHS of equation \eqref{EquivalenthfAssociativityLemEq2} 
lies in $V_h ((z))((w))$ while the RHS lies in $V_h ((w))((z))$.
Hence, one can multiply both sides by ${f_i}_k z^k$ and sum over $i=1, \ldots, s$ 
and $k \geq K$, obtaining \eqref{EquivalenthfAssociativity}.
The converse follows in a similar way because, for suitable 
$a'_j, b'_j \in V$ and ${f'_j}_{k} \in \mb{K}$, one has
\begin{displaymath}
\mc{S}^{-1}(z)(a \otimes b) = 
\sum_{j=1}^{s'} \sum_{k \geq K'} a'_{j} \otimes b'_j {f'_j}_{k} z^k\ \textrm{mod}\ h^k.
\end{displaymath}
Thus, one can do similar operations 
starting with \eqref{EquivalenthfAssociativity} applied to the elements 
$a'_{j} \otimes b'_j\otimes c$.
\end{proof}

\begin{lemma}
\label{LiLem.2.1b}
Let $V=W[[h]]$ be a topologically free $\mb K[[h]]$-module.
Let $a(z,w) \in W((z))((w))[[h]]$, $b(z,w) \in W((w))((z))[[h]]$, $c(z,w) \in W((w))((z))[[h]]$.
Then equation \eqref{LiJacobi} is equivalent to 
equations \eqref{maratona} modulo $h^M$ for every $M$
(with $N$ possibly depending on $M$).
\end{lemma}
\begin{proof}
Let us first prove the ``only if'' part. As $b(z,w) \in W((w))((z))[[h]]$, mod $h^M$ there is a suitable integer power $N$ of $z$ such that $z^N b(z,w)$ has only nonnegative powers of $z$. Multypling both sides of equation \eqref{LiJacobi} and taking the residue $Res_z$ we have the following equality:
\begin{align*}
Res_z z^N \gse{z}{w}\delta(x,z-w)\ a(z,w) = Res_z z^N \gse{z}{x} \delta(w,z-x)\ c(x,w)\, \text{ mod } h^M.
\end{align*}
The LHS is then equal to $Res_z z^N \gse{x}{w}\delta(x+w,z)\ a(z,w)$ which is in turn equal to $(x+w)^N \gse{x}{w} a(x+w,w)$. The RHS is equal to $Res_z z^N \gse{w}{x} \delta(w+x,z)\ c(x,w)$ which is in turn equal to $(w+x)^N c(x,w)$. We, thus have the second equation of \eqref{maratona} mod $h^M$ for any $M$. With the same reasoning line, we obtain the first equation of \eqref{maratona} using that $c(z,w) \in W((w))((z))[[h]]$.\\
Let us now focus on the ``if'' part. By the first equation of \eqref{maratona}, for $N \gg 0$, one has the following equalities:
\begin{equation}
\label{eq:ecografia1}
\begin{split}
&\gse{z}{w}\delta(x,z-w)\ (z-w)^N\ a(z,w) - \gse{w}{z} \delta(x,z-w)\ (z-w)^N\ b(z,w)\\
&= (\gse{z}{w} - \gse{w}{z}) \delta(x,z-w)\ (z-w)^N\ a(z,w)\ \text{ mod } h^M.
\end{split}
\end{equation}
Recalling the definition of the formal $\delta$ distribution and that $(\gse{z}{w} - \gse{w}{z})(z-w)^m = 0$ if $m \geq 0$ and $\frac{1}{(-n-1)!} \spd{w}^{-n-1} \delta(z,w)$ if $n < 0$, one has that
\begin{equation}
\label{eq:ecografia2}
\begin{split}
(\gse{z}{w}& - \gse{w}{z}) \delta(x,z-w)= \sum_{m \in \mb Z} x^{-m-1} (\gse{z}{w} - \gse{w}{z})(z-w)^{m}\\
&=\sum_{m < 0} x^{-m-1} \frac{1}{(-m-1)!} \spd{w}^{-m-1} \delta(z,w) = e^{x \spd{w}} \delta(z,w)\\
&= \gse{w}{x} \delta(z, w+x).
\end{split}
\end{equation}
Therefore, by equations \eqref{eq:ecografia1} and \eqref{eq:ecografia2}, one has
\begin{equation}
\label{eq:ecografia3}
\begin{split}
&\gse{z}{w}\delta(x,z-w)\ (z-w)^N\ a(z,w) - \gse{w}{z} \delta(x,z-w)\ (z-w)^N\ b(z,w)\\
& = \gse{w}{x} \delta(z, w+x) (z-w)^N\ a(z,w)\ \text{ mod } h^M.
\end{split}
\end{equation}
Using again the first part of equation \eqref{maratona}, one notes that $(z-w)^N\ a(z,w) \in W((w))((z))[[h]]$, hence, $z^L (z-w)^N\ a(z,w)$ has only non-negative integer powers of $z$ mod $h^M$ for $L \gg 0$. It follows that, multiplying both sides of equation \eqref{eq:ecografia3} by $z^L$ for $L \gg 0$, the following equalities hold:
\begin{align*}
&\gse{z}{w}\delta(x,z-w)\ z^L\ (z-w)^N\ a(z,w) - \gse{w}{z} \delta(x,z-w)\ z^L\ (z-w)^N\ b(z,w)\\
&= \gse{w}{x} \delta(z, w+x)\ z^L\ (z-w)^N\ a(z,w)\\
&= \gse{w}{x} \delta(z, w+x)\ (w+x)^L\ x^N\ a(w+x,w)\\
&= \gse{z}{x} \delta(z-x, w)\ \gse{w}{x} (w+x)^L\ x^N\ a(w+x,w)\\
&= \gse{z}{x} \delta(z-x, w)\ \gse{x}{w} (w+x)^L\ x^N\ a(w+x,w) \text{ mod } h^M.
\end{align*}
Using the second equation of \eqref{maratona} and recalling that $\delta(x,t)\ t^N = \delta(x,t)\ x^N$, one then has
\begin{align*}
&\gse{z}{w}\delta(x,z-w)\ z^L\ x^N\ a(z,w) - \gse{w}{z} \delta(x,z-w)\ z^L\ x^N\ b(z,w)\\
&= \gse{z}{x} \delta(z-x, w)\ z^L\ x^N\ c(x,w) \text{ mod } h^M
\end{align*}
from which, multiplying both sides by $z^{-L} x^{-N}$, one has
\begin{equation}
\label{eq:ecografia4}
\begin{split}
&\gse{z}{w}\delta(x,z-w)\ a(z,w) - \gse{w}{z} \delta(x,z-w)\ b(z,w)\\
&= \gse{z}{x} \delta(z-x, w)\ c(x,w) \text{ mod } h^M.
\end{split}
\end{equation}
Since equation \eqref{eq:ecografia4} holds for any $M$, equation \eqref{LiJacobi} follows.
\end{proof}

\begin{proof}[Proof of Theorem \ref{associativeBVABorcherdsIdentities}]
Let $(V, \vac, Y, \mc S)$ be a quantum vertex algebra. 
%
%By $\mc{S}$-locality \eqref{S-locality}, 
%for any $M \in \mb Z_{\geq0}$, there exists $N \geq 0$, such that
%\begin{equation}\label{eq:slocality}
%\begin{split}
%(z-w&)^N Y(z) \big( 1\otimes Y(w) \big)\big(\mc{S}(z-w)(a \otimes b) \otimes c \big)\\
%&=(z-w)^N Y(w) \big( 1\otimes Y(z) \big)(b \otimes a \otimes c) \quad \textrm{mod}\ h^M.
%\end{split}
%\end{equation}
By Lemma \ref{EK5Lem.1.2} we have $Y\mc S=Y^{op}$,
and then, by Lemma \ref{EquivalenthfAssociativityLem},
equation \eqref{EquivalenthfAssociativity} holds.
%to get that, for any $a, b, c \in V$ 
%and $M \in \mb Z_{\geq0}$, there exists $N' \geq 0$ such that
%\begin{equation}\label{eq:slocality2}
%\begin{split}
%(z+w)^{N'} &\gse{z}{w} Y(z+w) \big( 1 \otimes Y(w) \big) 
%\big( \mc{S}(z)(a \otimes b) \otimes c \big)\\
%&= (z+w)^{N'} Y(w) \big( Y^{op}(z) \otimes 1 \big) 
%(a \otimes b \otimes c)\ \textrm{mod}\ h^M.
%\end{split}
%\end{equation}
Let 
\begin{equation}\label{maratona6}
\begin{split}
a(z,w)&=Y(z) \big( 1\otimes Y(w) \big)\big(\mc{S}(z-w)(a \otimes b) \otimes c \big)
\,,\\
b(z,w)&=Y(w) \big( 1\otimes Y(z) \big)(b \otimes a \otimes c)
\,\text{ and }\\
c(z,w)&=Y(w) \big( Y(z)\mc S(z) \otimes 1 \big) (a \otimes b \otimes c)
\,.
\end{split}
\end{equation}
Note that, with these choices,
the $\mc S$-locality equation \eqref{S-locality}
is, modulo $h^M$, the first equation in \eqref{maratona},
while equation \eqref{EquivalenthfAssociativity} is, again modulo $h^M$,
the second equation in \eqref{maratona}
(with $z$ in place of $x$).
Therefore, by Lemma \ref{LiLem.2.1}, 
we get
\begin{align*}
\gse{z}{w} \delta(x, &z-w) Y(z) \big( 1\otimes Y(w) \big)
\big(\mc{S}(z-w)(a \otimes b) \otimes c \big)\\
- \gse{w}{z}& \delta(x, z-w) Y(w) \big( 1\otimes Y(z) \big)(b \otimes a \otimes c)\\
&= \gse{z}{x} \delta(w, z-x) Y(w) \big( Y(x)\mc S(x) \otimes 1 \big) 
(a \otimes b \otimes c)\ \textrm{mod}\ h^M.
\end{align*}
Since $V$ is a topologically free $\Kh$-module, this equation holds identically.
Since $Y\mc S=Y^{op}$, we thus have
\begin{equation}
\label{qJacobiop}
\begin{split}
\gse{z}{w} \delta(x, &z-w) Y(z) \big( 1\otimes Y(w) \big)\big(\mc{S}(z-w)
(a \otimes b) \otimes c \big)\\
- \gse{w}{z}& \delta(x, z-w) Y(w) \big( 1\otimes Y(z) \big)(b \otimes a \otimes c)\\
&= \gse{z}{x} \delta(w, z-x) Y(w) \big( Y^{op}(x) \otimes 1 \big) (a \otimes b \otimes c).
\end{split}
\end{equation}
Multiplying equation \eqref{qJacobiop} by $x^n$ and taking the residue $Res_x$, 
the LHS becomes
\begin{equation}\label{maratona2}
\begin{split}
& \gse{z}{w} (z-w)^n Y(z) (1\otimes Y(w)) \mc{S}^{12}(z-w)(a \otimes b \otimes c) \\
&- \gse{w}{z} (z-w)^n Y(w) (1\otimes Y(z))(b \otimes a \otimes c) 
\end{split}
\end{equation}
and the RHS becomes
\begin{equation}\label{maratona3}
\begin{split}
&Res_x \Big( x^n \gse{z}{x} \delta(w, z-x) Y(w) \big( Y^{op}(x) \otimes 1 \big) 
(a \otimes b \otimes c) \Big)\\
&= Res_x \Big( x^n e^{-x \spd{z}} \delta(z,w) Y(w) \big( Y^{op}(x) \otimes 1 \big) 
(a \otimes b \otimes c) \Big)\\
%&= Res_x \Big( x^n e^{-x \spd{z}} \delta(z, w) Y(w) \big( Y^{op}(x) \otimes 1 \big) 
%(a \otimes b \otimes c) \Big)\\
&= \sum_{l \geq 0} \frac{(-1)^l \spd{z}^l \delta(z, w)}{l!} Y(w) 
\big( a_{(n+l)}^{\mc{S}} b \otimes c \big)\\
&= \sum_{l \geq 0} Y(w) \big( a_{(n+l)}^{\mc{S}} b \otimes c \big) 
\frac{\spd{w}^l \delta(z, w)}{l!}.
\end{split}
\end{equation}
Combining \eqref{maratona2} and \eqref{maratona3},
we get the quantum Borcherds identity \eqref{qBorcherdsIdentities}.

Conversely, 
let $(V, \vac, Y, \mc S)$ be a braided state-field correspondence,
and assume that the quantum Borcherds identity \eqref{qBorcherdsIdentities} holds.
Multiplying it by $x^{-n-1}$ 
and summing over $n \in \mathbb{Z}$, 
the LHS becomes
\begin{equation}\label{maratona4}
\begin{split}
\gse{z}{w} \delta(x, &z-w) Y(z) \big( 1\otimes Y(w) \big)
\big(\mc{S}(z-w)(a \otimes b) \otimes c \big)\\
- \gse{w}{z}& \delta(x, z-w) Y(w) \big( 1\otimes Y(z) \big)(b \otimes a \otimes c)
\end{split}
\end{equation}
and the RHS becomes
\begin{equation}\label{maratona5}
\begin{split}
&\sum_{n \in \mathbb{Z}} x^{-n-1} \sum_{l \geq 0} Y(w) 
\big( a_{(n+l)}^{\mc{S}} b \otimes c \big) \frac{\spd{w}^l \delta(z, w)}{l!}\\
&= \sum_{l \geq 0} Y(w) \Big( \sum_{n \in \mathbb{Z}} 
a_{(n+l)}^{\mc{S}} b\ x^{-n-l-1} \otimes c \Big) \frac{x^l \spd{w}^l \delta(z, w)}{l!}\\
%&= \sum_{l \geq 0} Y(w) (Y^{op}(x)\otimes1) (a \otimes b\otimes c) 
%\frac{x^l \spd{w}^l \delta(z, w)}{l!}\\
&= \gse{z}{x} \delta(w, z-x) Y(w) \big( Y(x)S(x) \otimes 1 \big) (a \otimes b \otimes c).
\end{split}
\end{equation}
Equating \eqref{maratona4} and \eqref{maratona5},
we get equation \eqref{LiJacobi} with $a=a(z,w), b=b(z,w),c=c(z,w)$ as in \eqref{maratona6}.
By Lemma \ref{LiLem.2.1}, 
since $Y(z)(a \otimes c) \in V_h((z))$ 
and $Y(x)(b \otimes a) \in V_h((x))$, the associativity relation \eqref{hfAssociativity} 
and the $\mc{S}$-locality hold
on any vector. 
In particular, we have the $\mc{S}$-locality on the vacuum. 
By Remark \ref{EK5Lem.1.2a} we then have $Y\mc{S} = Y^{op}$.
The $\mc{S}$-locality now follows by Theorem \ref{July2017Thm}.
\end{proof}

Within the proof of Theorem \ref{associativeBVABorcherdsIdentities}
we have proved the following.
\begin{corollary}\label{cor:3:36}
Equation \eqref{qJacobiop} is equivalent to 
the Borcherds identity \eqref{qBorcherdsIdentities}.
\end{corollary}

We also have the following:
\begin{lemma}\label{lem:3:36}
Equation \eqref{qJacobiop} is equivalent to the $\mc S$-Jacobi identity
(which first appeared in \cite{L10}):
\begin{equation}
\label{LiSJacobi}
\begin{split}
& \gse{z}{w} \delta(x, z-w) Y(z) \big( 1\otimes Y(w) \big)(a \otimes b \otimes c)\\
& - \gse{w}{z} \delta(x, z-w) Y(w) \big( 1\otimes Y(z) \big)\big(\mc{S}(w-z)(b \otimes a) \otimes c \big)\\
& = \gse{z}{x} \delta(w, z -x) Y(w) \big( Y(x) \otimes 1 \big) (a \otimes b \otimes c).
\end{split}
\end{equation}
\end{lemma}
\begin{proof}
Exchanging, in equation \eqref{qJacobiop}, the variables $a$ and $b$ 
and the parameters $z$ and $w$, 
and changing sign to $x$, 
the LHS becomes
\begin{align*}
& \gse{z}{w} \delta(x, z-w) Y(z) \big( 1\otimes Y(w) \big)(a \otimes b \otimes c)\\
& - \gse{w}{z} \delta(x, z-w) Y(w) \big( 1\otimes Y(z) \big)\big(\mc{S}(w-z)(b \otimes a) \otimes c \big)
\,,
\end{align*}
since $\delta(-x,-y)=-\delta(x,y)$,
while the RHS becomes
\begin{align*}
&\gse{w}{x} \delta(z, w+x) Y(z) \big( Y^{op}(-x) \otimes 1 \big) (b \otimes a \otimes c)\\
&= \gse{z}{x} \delta(z -x, w) Y(z) \big( e^{-xT}Y(x) \otimes 1 \big) (a \otimes b \otimes c)\\
&= \gse{z}{x} \delta(w, z -x) Y(z-x) \big(Y(x) \otimes 1 \big) (a \otimes b \otimes c)\\
&= \gse{z}{x} \delta(w, z -x) Y(w) \big( Y(x) \otimes 1 \big) (a \otimes b \otimes c)
\,.
\end{align*}
For the first equality we used the definition of $Y^{op}$ 
%and the obvious identity $\gse{w}{x} \delta(z, w+x)=\gse{z}{x} \delta(z -x, w)$,
while, for the second equality, we used the translation covariance of $Y$.
\end{proof}

%%%
\subsection{The characterization Theorem}
\label{sec:5.5}

Combining the previous results, we have proved the following theorem:
\begin{theorem}\label{thm:final}
Let $(V, \vac, Y, \mc S)$ be a braided state-field correspondence.
The following statements are equivalent:
\begin{enumerate}[(i)]
\item $(V, \vac, Y, \mc S)$ is a quantum vertex algebra;
\item the $\mc{S}$-Jacobi identity \eqref{LiSJacobi} holds;
\item the associativity relation \eqref{hfAssociativity} and the equation $Y\mc{S} = Y^{op}$ hold;
\item the quantum Borcherds identity \eqref{qBorcherdsIdentities} holds;
\item the quantum $n$-product identities \eqref{qnproductIdentities} 
and the $\mc{S}$-locality \eqref{S-locality} hold.
\end{enumerate}
\end{theorem}
\begin{proof}
The equivalence of conditions (i) and (iii) was proved in Theorem \ref{July2017Thm}.
The equivalence of (i) and (iv) was proved 
in Theorem \ref{associativeBVABorcherdsIdentities}.
The equivalence of (ii) and (iv) is a consequence 
of Corollary \ref{cor:3:36} and Lemma \ref{lem:3:36}.
Finally, the equivalence of (iv) and (v) was proved in Lemma 
\ref{EquivalenceqBorcherdsIdentitiesqnprodIdentitiesLem}.
\end{proof}


\begin{thebibliography}{90}
\bibitem[BK]{BK}
B. Bakalov, V.G. Kac, \emph{Field algebras}, Int. Math. Res. Not., 3 (2003), 123-159.
\bibitem[BJK]{BJK}
M. Butorac, N. Jing, S. Kozic, 
\emph{h-adic quantum vertex algebras associated with rational R-matrix in types B, C and D}, 
preprint arXiv:1904.03771
\bibitem[B]{B}
R. Borcherds. \emph{Vertex algebras, Kac-Moody algebras, and the Monster}, 
Proc. Natl. Acad. Sci. USA 83 (1986), 3068-3071.
\bibitem[DSK06]{DSK06}
A. De Sole, V. Kac, \emph{Finite vs. affine W-algebras}, Jpn. J. Math., 1 (2006), 137-261.
\bibitem[EK5]{EK5}
P. Etingof, D. Kazhdan, \emph{Quantization of Lie bialgebras, V}, 
Selecta Math. 6, n.1 (2000), 105-130.
\bibitem[FLM]{FLM}
I. Frenkel, J. Lepowsky, A. Meurman, \emph{Vertex operator algebras and the Monster},
Academic Press, New York, (1988).
\bibitem[FR]{FR}
E. Frenkel, N. Reshetikhin, \emph{Towards deformed chiral algebras}, 
preprint arXiv:q-alg/9706023
\bibitem[Gar]{Gar}
M. Gardini, \emph{Quantum vertex algebras}, Ph.D Thesis (2019).
\bibitem[G]{G}
P. Goddard, \emph{Meromorphic conformal field theory}, 
in Infinite-dimensional Lie algebras and groups, 
Adv. Ser. in Math. Phys. 7 (1989), 556-587.
\bibitem[JKMY]{JKMY}
N. Jing, S. Kozic, A. Molev and F. Yang, 
\emph{Center of the quantum affine vertex algebra in type A}, 
J. Alg., 496 (2018), 138-186. 
\bibitem[K]{K}
V. Kac, \emph{Vertex algebras for beginners}, University lecture series, AMS, 10 (1996, II edition 1998).
\bibitem[K15]{K15}
V. Kac., \emph{Introduction to vertex algebras, Poisson vertex algebras, and integrable Hamiltonian PDE}, 
in Perspectives in Lie Theory, Springer INDAM series, 19 (2017), 3-72.
\bibitem[Kas]{Kas}
C. Kassel, \emph{Quantum groups}, Vol. 155. Springer Science \& Business Media, (2012).
\bibitem[L03]{L03}
H. Li, \emph{Axiomatic $G_1$-vertex algebras}, Commun. Contemp. Math. 5 n.2 (2003), 281-327. 
\bibitem[L10]{L10}
H. Li, \emph{h-adic Quantum Vertex Algebras and Their Modules}, 
Comm. Math. Phys., 296 n.22 (2010), 475-523.
\bibitem[LL]{LL}
J. Lepowsky, H. Li, \emph{Introduction to vertex operator algebras and their representations}, 
Springer Science \& Business Media, Vol. 227 (2004).
\end{thebibliography}
\end{document}